\documentclass{amsart}

\usepackage{amsmath}
\usepackage{amssymb}
\usepackage{amscd}
\usepackage{graphics}
\usepackage{latexsym}
\usepackage{hyperref}
\usepackage{xypic}
\usepackage{graphics}
\usepackage{hyperref}
\usepackage{tikz}

\newtheorem{theorem}{Theorem}[section]
\newtheorem{thm}[theorem]{Theorem}
\newtheorem{cor}[theorem]{Corollary}
\newtheorem{lem}[theorem]{Lemma}
\newtheorem{prop}[theorem]{Proposition}

\theoremstyle{definition}
\newtheorem{defn}[theorem]{Definition}

\newtheorem{ques}[theorem]{Question}

\newtheorem{rem}[theorem]{Remark}

\newtheorem{hyp}[theorem]{Hypothesis}

\newtheorem{conj}[theorem]{Conjecture}

\newtheorem{Observation}[theorem]{Observation}

\theoremstyle{remark}

\newcommand{\mbb}{\mathbb}

\newcommand{\ZZ}{\mbb{Z}}
\newcommand{\CC}{\mbb{C}}
\newcommand{\RR}{\mbb{R}}

\newcommand{\PP}{\mbb{P}}

\newcommand{\mc}{\mathcal}

\newcommand{\mcB}{\mc{B}}
\newcommand{\mcC}{\mc{C}}

\newcommand{\mcF}{\mc{F}}

\newcommand{\mcH}{\mc{H}}

\newcommand{\mcM}{\mc{M}}
\newcommand{\mcN}{\mc{N}}
\newcommand{\mcO}{\mc{O}}
\newcommand{\mcP}{\mc{P}}

\newcommand{\mcR}{\mc{R}}
\newcommand{\mcS}{\mc{S}}
\newcommand{\mcT}{\mc{T}}
\newcommand{\mcU}{\mc{U}}
\newcommand{\mcV}{\mc{V}}

\newcommand{\mcX}{\mc{X}}
\newcommand{\mcY}{\mc{Y}}

\newcommand{\HH}{\mcH(x,y)}

\newcommand{\mfm}{\mathfrak{m}}

\newcommand{\OO}{\mc{O}}

\newcommand{\wht}{\widehat}
\newcommand{\whts}{\wht{s}}

\newcommand{\whtOO}{\wht{\OO}}

\newcommand{\hk}{\hat{K}}

\newcommand{\pis}[1]{\text{Pseudo}_{#1}}

\newcommand{\SP}{\text{Spec }}

\newcommand{\p}{\mathbb{P}^{1}}

\newsavebox{\sembox}
\newlength{\semwidth}
\newlength{\boxwidth}

\newcommand{\Sem}[1]{%
\sbox{\sembox}{\ensuremath{#1}}%
\settowidth{\semwidth}{\usebox{\sembox}}%
\sbox{\sembox}{\ensuremath{\left[\usebox{\sembox}\right]}}%
\settowidth{\boxwidth}{\usebox{\sembox}}%
\addtolength{\boxwidth}{-\semwidth}%
\left[\hspace{-0.3\boxwidth}%
\usebox{\sembox}%
\hspace{-0.3\boxwidth}\right]%
}

\newsavebox{\semrbox}
\newlength{\semrwidth}
\newlength{\boxrwidth}

\newcommand{\Semr}[1]{%
\sbox{\semrbox}{\ensuremath{#1}}%
\settowidth{\semrwidth}{\usebox{\semrbox}}%
\sbox{\semrbox}{\ensuremath{\left(\usebox{\semrbox}\right)}}%
\settowidth{\boxrwidth}{\usebox{\semrbox}}%
\addtolength{\boxrwidth}{-\semrwidth}%
\left(\hspace{-0.3\boxrwidth}%
\usebox{\semrbox}%
\hspace{-0.3\boxrwidth}\right)%
}

\title
{Weak approximation for isotrivial families}

\author[Tian, Zong]{Zhiyu Tian,~~Hong Runhong Zong}
\address{
Mathematics 253-37\\
California Institute of Technology \\
Pasadena, CA, 91125}
\email{tian@caltech.edu}

\address{
Department of Mathematics\\
Princeton University \\
Princeton, NJ, 08544-1000}
\email{rzong@math.princeton.edu}

\date{\today}

\begin{document}


\begin{abstract}
 We prove weak approximation for isotrivial families of rationally connected varieties defined over the function field of a smooth projective complex curve.
\end{abstract}


\maketitle



\section{Introduction}
Let $X$ be a variety over a number field or function field $K$, it has been several decades both for number theorists and geometers to investigate into the set of rational points $X(K)$; when is $X(K)$ non-empty? and how many of them if $X(K)$ is non-empty?~etc. Elementary observations suggest a good candidate class of varieties where the most rational points emerge as the \emph{rational varieties}, which belong to a broader class as \emph{rationally connected} varieties. A variety $X$ is \emph{rationally connected} if for two general points $x, y \in X$, there is a rational curve connecting them, e.g.~there is an irreducible component $V$ of $Hom(\p, X)$ and a family of maps $f: \p \times V \rightarrow X$ such that the double evaluation $f^{(2)}:\p \times \p \times V \rightarrow X\times X$ is dominant. For a characteristic free version, to be \emph{separably rationally connected} is to further require $f^{(2)}$ generic \'etale.  We recall the pioneering work by Prof.~Graber, Prof.~Harris and Prof.~Starr (\cite{GHS03})
\begin{thm}\label{ghs} Let $\pi:\mcX \rightarrow B$ be a flat surjective morphism from a projective variety to a smooth projective complex curve such that a general fiber is smooth and rationally connected. There $\pi$ has a section $s: B \to \mcX$.
\end{thm}
Denote the generic fiber of $\mcX \to B$ by $\mcX_{\eta}$ and the function field of $B$ by $K=\mathbb{C}(B)$. Theorem \ref{ghs} said that $\mcX_{\eta}(K)=\mcX_{\eta}(\mathbb{C}(B)) \neq \emptyset$. Furthermore, let $\hat{\OO}_{B,b}$ be the completion of the local ring at $b \in B$ and let $Frac\ \hat{\OO}_{B,b}$ denote its fraction field. One can consider the ring of \emph{Ad\'eles} $$\mathbb{A}(B):= \prod^{o}_{b \in B} Frac \ \hat{\OO}_{B,b}$$ of $\mathbb{C}(B)$ with weak product topology. Prof.~Hassett and Prof.~Tschinkel proposed the following conjecture in \cite{HT06}
\begin{conj}\label{HT}
Notation as in Theorem \ref{ghs}. Then $\mcX \to B$ satisfies \emph{weak approximation} at all places, namely the embedding $\mcX_{\eta}(K)=\mcX_{\eta}(\mathbb{C}(B)) \subset \mcX_{\eta}(\mathbb{A}(B))$ has dense image in $\mcX_{\eta}(\mathbb{A}(B))$. Equivalently, for every finite sequence
$(b_1,\dots,b_m)$ of distinct closed points of $B$, for every sequence
$(\whts_1,\dots,\whts_m)$ of formal power series sections of $\pi$
over $b_i$, and for every positive integer $N$, there exists a regular
section $s$ of $\pi$ which is congruent to
$\whts_i$ modulo $\mfm_{B,b_i}^{N}$ for every $i=1,\dots,m$.
\end{conj}

Some special cases of the conjecture are known, e.g.
\begin{itemize}
\item $\PP^n$, conic bundles over $\PP^1$, del Pezzo surfaces of degree at least $4$,
\item low degree complete intersections of degree $(d_1, \ldots, d_c)$ such that $\sum d_i^2 \leq n+1$ \cite{WASurvey},
\item smooth cubic hypersurfaces in $\PP^n, n\geq 6$ \cite{WAhypersurface},
\item at places of good reduction (for any family) \cite{HT06},
\item a general family of del Pezzo surfaces of degree at most $3$, \cite{BadReduction}, \cite{Knecht}, \cite{XuWA},
\item a smooth hypersurface with square-free discriminant \cite{WAhypersurface}.
\end{itemize}
And the starting point of the current work is the the following question by Prof.~Starr

\begin{ques}\label{obs}
Assume $k=\bar{k}$. Let $X$ be a smooth projective separably rationally connected $k$-variety, and $G$ be a cyclic subgroup of $Aut(X)$. Then is $X$ $G$-equivariantly rationally connected? Namely, for a pair $x, y$ as fixed points of $G$, is there a $G$-equivariant rational curve connecting them?
\end{ques}


Over complex numbers, Question \ref{obs} follows from Conjecture \ref{HT} so can be seen as a geometric obstruction to the latter(c.f.~Theorem \ref{gob}). Now recalling the proof of Theorem \ref{ghs} in \cite{GHS03}: for a rationally connected fibration $\mcX \to B$, by the powerful smoothing of comb argument initiated by Prof.~Koll\'ar, Prof.~Miyaoka and Prof.~Mori in \cite{KMM92RC}, and a specialization argument cancelling monodromy of $C^*\to B$ around multiple fibers, one can find a \emph{``flexible''} curve $C^*$ where the forgetful map $$ {\mcF}_{g,0}: \overline{\mcM}_{g,0}(\mcX, [C^*]) \to \overline{\mcM}_{g,0}(B, \pi_{*}[C^*])$$ is smooth and surjective. Degenerate $(C^* \to B)$ in the Hurwicz scheme $\mcH_{g,B}=\overline{\mcM}_{g,0}(B, \pi_{*}[C^*])$ to contain a component isomorphic to $B$, a preimage of this component will be a section. Quite unexpectedly, we make an \emph{Observation} that the situation of Question \ref{obs} inherits a variant of the main argument of in \cite{GHS03}; based on a $G$-equivariant curve $f: C \hookrightarrow X$ connecting $x$ to $y$ we consider a moduli compactification of the rational \emph{forgetful} map$$\xymatrixcolsep{2pc}\xymatrixrowsep{0.8pc}\xymatrix{{\mcF}^G_{g,2}: \overline{\mcM}^{G}_{g,2}(X, [C])\{f(x_1)=x, f(x_2)=y \} \ar@{.>}[r]& \mcH_{g,2}(G)\\**[l] (f: C\hookrightarrow X)\in \mcU  \ar[u]  \ar[ur]}$$-the left hand side is the $G$-fixed point set, and an open neighbourhood $\mcU$ of $C$ in it parametrizes $G$-equivariantly embedded curves connecting $x$ to $y$-so it maps to ${\mcH}_{g,2}(G)$ as the Hurwicz scheme of Galois covers $C\to C/G$ with two specified ramifications. Taking a good functorial compactification of ${\mcF}^G_{g,2}$, then similar to \cite{GHS03}, tracing back the forgetful map, the preimage of a good component of some degenerated $G$-cover of $C \to C/G$ will give a $G$-equivariant rational curve connecting $x$ to $y$. We note that for the compactification, theory of \emph{twisted stable maps} by Prof.~Abramovich, Prof.~Olsson and Prof.~Vistoli is used, and an explicit pencil of twisted curves cancels the monodromy of $C \to C/G$ which is similar to the second main construction in \cite{GHS03}. This finally leads to

\begin{thm}\label{thm:Hypersurface}
Notation as in Question \ref{obs}. Assume that $G \cong \mathbb{Z}/l \mathbb{Z}$ with $l$ divisible in $k$, and that $G$ acts on $\mathbb{P}^1$ by the canonical action $z \mapsto \zeta z$, where $\zeta$ is a primitive $l$-th root of unity. Then there is a $G$-equivariant map $f: \mathbb{P}^1 \to X$ such that $f(0)=x$ and $f(\infty)=y$.
\end{thm}

Back to Conjecture \ref{HT}. Though the vanishing of this geometric obstruction does not necessarily implies Conjecture \ref{HT}, we prove that it already suffices for isotrivial families; a family $\mcX \to B$ is isotrivial if there is an \'etale (but not necessarily surjective) morphism $B' \to B$ such that there is a $B'$ isomorphism
\[
\mcX'=\mcX \times_B B' \cong X \times B',
\]
for some fixed variety $X$. Our methods uses a relative version of $G$-equivariant smoothing of combs which works for families with $G$-actions. And more generally, our argument also works for families at places satisfying the following Hypothesis \ref{hyp}, which covers a large part of all the cases currently known for Conjecture \ref{HT}.

\begin{hyp}\label{hyp}
For any $b\in B$ with $\mcX_b$, let $K$ be the fraction field $Frac \ \whtOO_{b,B}$ of $\whtOO_{b,B}$ and $\mcX_K$ be the base change to $\SP{K}$. Assume there is a Galois extension $K'/K$ with cyclic Galois group $G_b$ such that $\mcX_K\times_{\SP{K}}\SP{{K}'}$ extends to a smooth proper family $\mcX' \to \SP{\whtOO'}$, where $\whtOO'$ is the ring of integers in $K'$. Further assume the action of $G_b$ extends to $X'$ and is compatible with the projection to $\SP \whtOO'$.
\end{hyp}

\begin{thm}\label{m}
Notation as in Theorem \ref{ghs}. Then Conjecture \ref{HT} holds for the family $\mcX \to B$ at places $b \in B$ which satisfy Hypothesis \ref{hyp}. In particular, Conjecture \ref{HT} holds for isotrivial families (Proposition \ref{lem:basechange}).
\end{thm}

An example satisfying the above hypothesis \ref{hyp} is the family of cubic surfaces $F(X_1, X_2, X_3)+tX_0^3+\ldots$, whose central fiber is a cone over a smooth elliptic curve. One can make a degree $3$ base change $t=s^3$ and change of variables
\[
Y_0=sX_0, Y_1=X_1, Y_2=X_2, Y_3=X_3.
\]
The new family has a $\ZZ/3 \ZZ$ action and the central fiber $F(Y_1, Y_2, Y_3)+Y_0^3$ is smooth. As proposed by Prof.~Hassett in \cite{WASurvey} this is expected to be an exceptionally hard case of Conjecture \ref{HT}. We remark that technics and ideas here, esp.~around this example, together with some observations and ideas which deal with much subtler singularities, finally led to the solution of Conjecture \ref{HT} for cubic hypersurfaces, esp.~for cubic surfaces in \cite{T2}. This is currently the best known low dimensional result to Conjecture \ref{HT}.

\begin{rem}In spite of the classical relative equivariant smoothing in Section 4 which culminates all the smoothing technics, we present another proof of Theorem \ref{m} in the Appendix, which uses some notion from derived algebraic geometry and is more conceptual and straightforward. Following the idea of lifting by forgetful maps in the proof of Theorem \ref{ghs} and our Theorem \ref{thm:Hypersurface}, for a general rationally connected fibration $\mcX \to B$ with bad reduction at $b$,  we use the \emph{derived} forgetful map initiated by Prof.~Starr and Roth $$\xymatrixcolsep{2pc}\xymatrix{
\mcF_D:\mathcal{H}ilb_{\mcX/B} \ar[r]& {\mcP}seudo_{\mcX_b}}$$ here ${\mcP}seudo_{\mcX_b}$ is the moduli stack of pseudo-ideal sheaves, or equivalently-differential graded subscheme of $\mcX_b$ with amplitude $(1,0)$. All sections lie in $\mc{H}ilb_{\mcX/B}$, and  $\mcF_D$ is the  \emph{derived} intersection product $$\xymatrixcolsep{2pc}\xymatrix{ C \ar@{|->}[r] & C \times_{\mcX}^{L} \mcX_b}$$ introduced by Prof.~Lurie (\cite{L04}, \cite{L11}). As one of the applications of \emph{derived Algebraic Geometry}, the prominent feature of $\mcF_D$ is that it has a well-behaved tangent obstruction $\mcR \Gamma(C,\mcN_{C/\mcX}\otimes \OO_{\mcX_b})$ at $C\subset X$ when $C$ is locally complete intersection-$\mcF_D$ remembers the jet datum of $(C \subset \mcX) \to B$ and so once the first term of $\mcR \Gamma(C,\mcN_{C/\mcX}\otimes \OO_{\mcX_b})$ is cancelled at some ``flexible'' curve, then a well-behaved degeneration and the forgetful-map type argument applies to get the approximating section as desired. To explore for wider situations than in Hypothesis \ref{hyp} where this machinery technically applies is believed by the authors to be a possible way towards a full solution to Conjecture \ref{HT}.
\end{rem}

\textbf{Acknowledgment:} The authors would first like to thank Professor Jason Starr for introducing them to the question and inumerous helpful discussions; Professor J\'anos Koll\'ar for his constant support of the second named author and encouraging comments on the proof of Theorem \ref{thm:Hypersurface}; Professor Chenyang Xu for helpful discussions, sharing us with his ideas and generous support in Beijing International Center of Mathematics Research where the authors partially carried out this project; Professor Tommaso de Fernex for helpful comments on Theorem \ref{m}; Professor Martin Olsson, Professor Vivek Shende, and Professor Xinyi Yuan for enlightening comments and suggestions on the first version of this paper; Doctor Qile Chen and Doctor Yi Zhu for reading part of the first draft and helpful comments.


\section{Preliminary results}\label{sec:pre-results}
In this section we collect some preliminary results for later reference. The results proved here are slightly stronger than what are actually needed. In this section we will assume that $G$ is a finite group except in Lemma \ref{lem:EquivSmoothing}, and $k$ is an algebraically closed field whose characteristic is not divisible by the order of $G$.

\subsection{Everything with a group action}

Firstly, we are concerned with the following infinitesimal lifting problem. Let $S$ and $R$ be $k$-algebras with a $G$-action and $f: S \rightarrow R$ be a $k$-algebra homomorphism compatible with the action. Also assume that $R$ is a finite type $S$-algebra via the map $f$. Let $A$ be an Artinian $k$-algebra with a $G$-action, $I \subset A$ an invariant ideal such that $I^2=0$. Consider the following commutative diagram, where $p$ is a $G$-equivariant $k$-algebra homomorphism
$$
\xymatrixcolsep{3pc}\xymatrix{
S\ar[r]^f \ar[d] & R \ar[d]^p\\
A\ar[r]^{\pi} & A/I \ar[r] &0
}
$$
We want to know when one can find a $G$-equivariant lifting $h:R\rightarrow A$. The following lemma completely answers this question.

\begin{lem}\label{lem-equiv-lifting}
If we can lift the map $p$ to a $k$-algebra homomorphism $h: R\rightarrow A$ such that $\pi\circ h=p$, then we can find an equivariant lifting $\tilde{h}:R \rightarrow A$ with the same property.
\end{lem}

\begin{proof}
For every element $g$ in $G$, define a map $h_g:R\rightarrow A$ by $h_g(r)=g\cdot h(g^{-1}\cdot r)$. This is an $S$-algebra homomorphism and also a lifting of the map $p: R\rightarrow A/I$. The map $h$ is $G$-equivariant if and only if $h_g(r)=h(r)$ for every $g\in G$ and every $r \in R$. The difference of any two such liftings is an element in $Hom(\Omega_{R/S}, I)$, where $\Omega_{R/S}$ is the module of relative differentials. Therefore one has $\theta(g)(r)=h_g(r)-h(r)$ in $Hom(\Omega_{R/S}, I)$. Notice that $Hom(\Omega_{R/S}, I)$ is naturally a $G$-module with the action of $G$ on $Hom(\Omega_{R/S}, I)$ given by
\[
G \times Hom(\Omega_{R/S}, I) \rightarrow Hom(\Omega_{R/S}, I)
\]
\[
(g, \eta) \mapsto g\cdot\eta=(\omega \mapsto g\cdot\eta(g^{-1}\cdot \omega)).
\]

It is easy to check that
\[
\theta(gh)=g\cdot\theta(h)+\theta(g).
\]
Thus $\theta$ defines an element $[\theta]$ in $H^1(G,Hom(\Omega_{R/S}, I))$. The existence of an equivariant lifting is equivalent to the existence of an element $\Theta \in \textsl{Hom}(\Omega_{R/S}, I)$ such that $g \Theta-\Theta=\theta$, i.e, the class defined by $\theta$ is zero in $H^1(G,Hom(\Omega_{R/S}, I))$. But the characteristic of the field is relatively prime to the order of $G$, so the higher cohomologies of $G$ vanish (\cite{weibel94}, Proposition 6.1.10, Corollary 6.5.9).
\end{proof}

\begin{cor}\label{sm-equiv-lifting}
Let $X$ and $Y$ be two $k$-schemes with a $G$-action and $f: X \rightarrow Y$ be a finite type morphism compatible with the action. Let $x \in X$ be a fixed point, and $y=f(x)$ (hence also a fixed point). Assume that $f$ is smooth at $x$. Then there exists a $G$-equivariant section $s: \SP\whtOO_{y,Y}\rightarrow X$.
\end{cor}

\begin{proof}
Let $S$ be the local ring at y, and $R$ be the local ring at x. There is an obvious $G$-action on both of these $k$-algebras. We start with the section
\begin{align*}
s_0: \SP k(y) &\rightarrow f^{-1}(y), \\
\SP k(y)&\mapsto x,
\end{align*}
 which is clearly $G$-equivariant. By the smoothness assumption, a section from $\SP(\whtOO_{y,Y}/{\mfm_{y}^n})$ always lifts to a section from $\SP(\whtOO_{y,Y}/{{\mfm_y}^{n+1}})$. Now apply Lemma~\ref{lem-equiv-lifting} inductively to finish the proof.
\end{proof}

\subsection{Comb and equivariant comb construction}

Let $X$ be an algebraic variety over an algebraically closed field $k$, a \emph{very free curve(map)} is an $f: \mathbb{P}^1 \to X$ with $f^{*}\mcT_X$ ample. Let $X_{sm}$ be the smooth locus of $X$, and \emph{very free locus} $X_{vf} \subset X_{sm}$ be the open locus of points $x \in X$ which lies in the image of a \emph{very free curve}. Then it is well known that $X_{sm}$ is separably rationally connected if and only if $X_{vf} \neq \emptyset$ (\cite{KMM92RC}, \cite{Kollar96}).
\begin{defn}
\label{defcomb}
Let $k$ be an arbitrary field. A {\it comb} with {\it $n$ teeth} over $k$ is a curve
$C$ with $n+1$ components $D,C_1,\dots,C_n$ over $\bar k$
satisfying:
\begin{enumerate}
\item $D$ and $C_1\cup\dots\cup C_n$ are defined over $k$
\item $C_i\cong \mathbb{P}^1$, $C_i\cap C_j=\emptyset$ for $i\neq j$, and each $C_i$ intersects $D$
transversely in a single smooth point over $\bar{k}$.
\end{enumerate}
$D$ is the {\it handle} of the comb, and $C_1,\dots,C_n$ the {\it teeth}.
\end{defn}

\begin{defn}[Smoothing]Given $f: C\to X$ from a {\it comb} $C$ with {\it handle} $D$, a smoothing of $f$ is a family $\Sigma \to T$ over a pointed curve $(T, 0)$ with $F: \Sigma \to X$ such that $\Sigma_0 \cong C, F\vert_{\Sigma_0}=f$ and $\Sigma_t$ is a smooth curve for general $t$.
\end{defn}
We have a general procedure for \emph{smoothing of comb} using \emph{very free curves}, the most general form is presented in \cite{TZ}.
\begin{thm}[Proposition 2.4 of \cite{TZ}]\label{smoothing}
Given $f_0: C_0 \to X$ with $C_0$ a curve and $X$ a smooth quasi-projective $k$-variety, and an integer $d$. Suppose $X$ is separably rationally connected and  $f_0(C) \cap X_{vf}\neq \emptyset$, then there are $q\gg0$ {\it very free curves} $f_i: C_i \to X, 1 \leq i \leq q$, such that:
\begin{enumerate}
 \item $C=C_0\cup C_1\cup...\cup C_q$ is a {\it comb} with $q$ {\it teeth}. And there is a $k$-morphism $f: C \to X$ with a smoothing $\Sigma \to T, G: \Sigma \to X$.
\item $H^1(\Sigma_t, G_t^*{\mcT_{X}}\otimes M)=0$ for a general member $G_t: \Sigma_t \to X$ and any line bundle $M$ with $|deg\ M|\leq d$.
\end{enumerate}
\end{thm}

For proper family $\mcX \to B$ where $\mcX$ can be a \emph{Deligne-Mumford} stack, we need a relative version of \emph{smoothing of comb} which can be proved using the same method for Theorem \ref{smoothing}. And it was implicitly stated in \cite{GHS03}, \cite{dj-st} and \cite{HT06}.

\begin{thm}[Relative smoothing]\label{relative smoothing}
Let $\pi: \mathfrak{X}\to B$ be a family of smooth quasi-projective $k$-varieties over a curve $B$ with separably rationally connected general fibers. Given: a compactification $\bar{\pi}: \overline{\mathfrak{X}} \to B$ where $\overline{\mathfrak{X}}$ can be an algebraic space, Deligne-Mumford stack or Artin stack; a multisection $(C' \subset \overline{\mathfrak{X}}) \to  B$ which lies in the smooth locus of $\bar{\pi}$ and intersects with $\bigcup_{b\in B} {\overline{\mathfrak{X}}_{b}}_{vf}$; an integer $d$. Then there are $q\gg0$ morphism $f_i: C_i\cong \mathbb{P}^1 \to \overline{\mathfrak{X}}, 1 \leq i \leq q$, such that:
\begin{enumerate}
\item For $ 1 \leq i \leq q$,  $g_i:C_i\rightarrow \overline{\mathfrak{X}}$ factors through a separably rationally connected fiber  $\overline{\mathfrak{X}}_{c_i}$ for some $c_i\in C$,  and $g_i:C_i\rightarrow \overline{\mathfrak{X}}_{c_i}$ is very free in $\overline{\mathfrak{X}}_{c_i}$, namely, $C_i$ lies in the smooth locus of $\bar{\pi}$ with $f_i^{*}\mcT_{\overline{\mathfrak{X}}/B}$ ample.
 \item $\hat{C}=C'\cup C_0\cup C_1\cup...\cup C_q$ is a {\it comb} with $q$ {\it teeth} with a natural $k$-morphism $f: C_t \to \overline {\mathfrak{X}}$ and a smoothing $\Sigma \to T, G: \Sigma \to \overline{\mathfrak{X}}$ which lies in the smooth locus of $\bar{\pi}$.
\item $H^1(\Sigma_t, G_t^*{\mcT_{\overline{\mathfrak{X}}/B}}\otimes M)=0$ for a general member $G_t: \Sigma_t \to \overline{\mathfrak{X}}$ and any line bundle $M$ with $|deg\ M|\leq d$, where $G_t: \Sigma_t \to \overline{\mathfrak{X}}$ denotes the restriction of $G$ to the fiber $\Sigma_t$.
\end{enumerate}
\end{thm}
\begin{proof} The proof is basically the same as \cite[Proposition 2.4 ]{TZ}, so we only sketch the proof.
By \cite[Lemma 25]{HT06} and \cite[Proposition 24]{HT06}, one can attach  sufficiently many very free curves $C_i,i=1,\ldots,q$ on the fibers of $\pi$ to $C'$ such that there is a smoothing of the comb $\hat{C}=C'\cup C_1\cup...\cup C_q$ whose handle is $C'$. Then there is a natural morphism $g:\hat{C}\rightarrow \overline{\mathfrak{X}}$.  Moreover, one can choose $q$ large enough such that $q-h^1(C', (g^\ast \mcT_{\overline{\mathfrak{X}}/B})|_{C'})\gg 0 $. Then $H^1(\Sigma_t, G_t^*{\mcT_{\overline{\mathfrak{X}}/B}}\otimes M)=0$ for  any line bundle $M$ on $\Sigma_t$ with $|deg\  M|$ bounded $d$ by \cite[Lemma 2.5]{TZ} and \cite[Lemma 2.6]{TZ}.
\end{proof}

We also need a $G$-equivariant version of Theorem \ref{smoothing}.
\begin{thm}[G-equivariant smoothing]\label{equivariant smoothing}
Assume $k=\bar{k}$. Let $C_0$ be a curve and $X$ a smooth quasi-projective $k$-variety both with $G$ actions, with a $G$-equivariant embedding $f_0: C_0 \hookrightarrow X$ and an integer $d$. Suppose $X$ is separably rationally connected and  $f_0(C) \cap X_{vf}\neq \emptyset$, then there are $q\gg0$ {\it very free curves} $f_i: C_i \to X, 1 \leq i \leq q$, such that:
\begin{enumerate}
 \item $C=C_0\cup C_1\cup...\cup C_q$ is a {\it comb} with $q$ {\it teeth} invariant by $G$.  And there is a $G$-equivariant morphism $f: C \to X$ with $G$-equivariant smoothing $\Sigma \to T, G: \Sigma \to X$, $T$ with trivial $G$ action.
\item $H^1(\Sigma_t, G_t^*{\mcT_{X}}\otimes M)^G=0$ for a general member $G_t: \Sigma_t \to X$ and any line bundle $M$ with a $G$ linearization and $|deg\ M|\leq d$.
\end{enumerate}
\end{thm}
\begin{proof}
One can assume $\dim X\geq 3$ by replacing $X$ with $X\times \PP^M, M\gg 0$. Let $C_0$ be the image of $f_0$. Apply Theorem \ref{smoothing}, pick $p\gg 0$ very free curves $C_i$ intersecting with $C$ transversely at general point $p_i$ with general tangent direction. Form the {\it comb} $C=C_0 \cup_{g\in G} \cup_{i=1,...,p} g\ C_i$ the latter is $G$-invariant and have $H^1(C, {\mcT}_X|_{C} \otimes M)=0$ for any line bundle $M$ with $|deg\ M|\leq d$. The vanishing implies that Since $|G|$ is divisible in $k$ we have $H^1(C, {\mcT}_X|_{C} \otimes M)^G=0$ since higher Galois cohomologies of $G$ vanish (c.f. beginning of this section and the proof of Lemma \ref{lem-equiv-lifting}). So it remains to prove the existence of a $G$-equivariant smoothing of $C$, which is equivalent to find a $G$-equivalent section of $H^0(C,\mcN_{C/X})$ which has general direction at the nodes of $C$ and hence can smooth them.

Let $D_1=C_0 \cup C_1 ...\cup C_p$, by Theorem \ref{smoothing} again, $\mcN_{D_1/X}$ is globally generated if $p\gg 0$ and $H^1(D_1, \mcN_{D_1/X} \otimes L_1)=0$, where $L_1$ is a line bundle on $D_1$ which has degree $-l$ on $C$ and $0$ on all the other irreducible components $C_i$'s. After attaching all $G$-conjugates of $C_i$'s, the new nodal curve is simply $C$. We have $H^1(C, \mcN_{C/X}\otimes L)=0$, where $L$ is an extension of the line bundle $L_1$ on $C_0$ which has degree $-l$ on $C_0$ and $0$ on all the other irreducible components.

We have the following two exact sequences
\[
0 \to \oplus \mcN_{C/X}|_{C_j}(-p_j) \to \mcN_{C/X} \to \mcN_{C/X}|_{C_0} \to 0,
\]
\[
0 \to \mcN_{C_0/X}\to \mcN_{C_0/X} \to \oplus_j Q_j \to 0,
\]
where $Q_j$ is a torsion sheave supported on the point $p_j \in C_0$. Every sheaf has a natural $G$ action and the $G$-equivariant deformations are given by $G$-invariant sections of $\mcN_{C/X}$. One just need to find a $G$-invariant section in $H^0(C, \mcN_{C/X})$ which is not mapped to $0$ under the composition of maps $$H^0(C, \mcN_{C/X}) \to H^0(C_0, \mcN_{C/X}|_{C_0}) \to Q_j$$ for all $j$.

Since $H^1(C, \mcN_{C/X}\otimes L)=0$, we also have $H^1(C_0, \mcN_{C/X}\otimes L\otimes \OO_{C_0})=0.$ Let $c_1, \ldots, c_l$ be an orbit of the $G$ action on $C_0$. Then there is a section of $\mcN_{C/X}|_{C_0}$ which vanishes on $c_1, \ldots, c_{l-1}$ but not on $c_l$. Taking average over $G$ gives a $G$-invariant section of $\mcN_{C/X}|_{C_0}$ which does not vanish on any of $c_1, \ldots, c_l$. In particular, for any $l$ nodes on $C_0$ which lie in a $G$-orbit, we can find a $G$-invariant section of $\mcN_{C/X}$ which does not vanish on them. Then a general $G$-invariant section of $\mcN_{C/X}|_{C_0}$ does not vanish on any of the nodes $p_i$'s. We have the surjection $$H^0(C, \mcN_{C/X}) \to H^0(C_0, \mcN_{C/X}|_{C_0}) \to 0$$ and hence the surjection $$H^0(C, \mcN_{C/X})^G \to H^0(C_0, \mcN_{C/X}|_{C_0})^G \to 0$$ by the vanishing of higher \emph{Galois} cohomologies of $G$ again. So a general $G$-invariant section in $H^0(C, \mcN_{C/X})^G$ does not vanish on the nodes. We take the $G$-equivariant deformation given by this section, which necessarily smooths all the nodes of $C$.
\end{proof}
\begin{rem}\label{coherent sheaf}
In Theorem \ref{smoothing}, Theorem \ref{relative smoothing} and Theorem \ref{equivariant smoothing}, by taking resolution of syzygies one can replace $M$ by a specific coherent sheaf supported on the \emph{comb}, and it is left to reader to find a universal degree bound.
\end{rem}

\subsection{Iterated blow-up}

Let $\pi: \mcX \to C$ be a flat proper family over a smooth projective connected curve $C$. Let $c \in C$ be a closed point and $\whts_0: \SP \whtOO_{c, C} \to \mcX$ be a formal section. Assume that $\whts_0$ lies in the smooth locus of $\mcX \to C$. The $N$-th iterated blow-up associated to $\whts_0$ is defined inductively as follows.

The $0$-th iterated blow-up $\mcX_0$ is $\mcX$ itself. Assume the $i$-th iterated blow-up $\mcX_i$ has been defined. And let $\whts_i$ be the strict transform of $\whts_0$ in $\mcX_i$. Then $\mcX_{i+1}$ is defined as the blow-up of $\mcX_i$ at the point $\whts_i(c)$.

We remark that if both $\mcX$ and $C$ has a $G$ action such that
\begin{itemize}
\item
the map $\pi: \mcX \to C$ is $G$-equivariant.
\item
the point $c$ is the fixed point of $G$. And $\whts_0$ is $G$-equivariant.
\end{itemize}
Then each $\mcX_i$ has a $G$ action such that both the natural morphism $\mcX_{i+1} \to \mcX_i$ and the formal section $\whts_i$ are $G$-equivariant. In particular, the intersection of $\whts_i$ with the central fiber is a fixed point of $G$.

One can also do this at fibers over a $G$-orbit in $C$, provided the formal sections over these points are mapped to each other via the $G$ action. Then the iterated blow-up still has a $G$ action and every morphism is compatible with the action.

On $\mcX_N$, the fiber over the point $c$ are the strict transform of $\mcX_c$ and exceptional divisors $E_1, \ldots, E_N$ and
\begin{itemize}
\item
$E_i, i = 1, \ldots, N-1$, is the blowup of $\PP^d$ at $r_i (=\whts_i(c))$, the point where
the proper transform of $\whts_0$ (i.e. $\whts_i$) meets the fiber over $c$ of the $(i-1)$-th iterated blow-up;
\item
$E_N \cong \PP^d$, where $d$ is the dimension of the fiber.
\end{itemize}

The intersection $E_i \cap E_{i+1}$ is the exceptional divisor $\PP^{d-1}\subset E_{i-1}$, and a proper transform of a hyperplane in $E_{i+1}$, for $i=0, \ldots, N-1$.

Furthermore, to find a section that agrees with $\whts_0$ up to the N-th order is the same as finding  a section in $\mcX_{N+1}$ intersecting the fiber over $c$ at $E_{N+1}$, or equivalently, a section in $\mcX_N$ which intersects the exceptional divisor $E_N$ at the point $r_N=\whts_N(c)$ (Proposition 11, \cite{HT06}).

\subsection{Isotrivial family and geometric obstruction}
\begin{defn}\label{iso}
Let $\pi: \mcX \to B$ be a flat proper family of $k$-varieties over arbitrary base $B$. It is isotrivial if there is an \'etale morphism $B' \to B$ such that there is a $B'$ isomorphism $\mcX'=\mcX \times_B B' \cong X \times B'$ for some $k$-variety $X$.
\end{defn}
We have to show that an isotrivial family over complex numbers satisfies Hypothesis \ref{hyp}, as follows:

\begin{prop}\label{lem:basechange}
Let $\mathbb{C}\Semr{t}$ be the Laurent fied over complex numbers. And let $\pi: \mcX \to \SP \CC\Semr{t}$ be an projective family of complex varieties. If $\pi$ is isotrivial,
then there is a finite cyclic group $G$ of order $l$, a smooth projective variety $X$, and a group homomorphism from $G$ to the automorphism group of $X$, such that the family $\mcX \to \SP \CC\Semr{t}$ is isomophic to the quotient $(X \times \SP \CC\Semr{t'})/G$, where the action of $G$ on $\SP \CC\Semr{t'}$ is given by the choice of a primitive $l$-th root of unity $\zeta$ and $t \mapsto \zeta \cdot t'$. In particular, Hypothesis \ref{hyp} is satisfied for any isotrivial family over complex numbers.
\end{prop}

\begin{proof}
By definition, after an \'etale base change, the family becomes trivial. Since the only connected \'etale cover of $\SP \CC\Semr{t}$ is of the form $$\SP \CC\Semr{t'} \to \SP \CC\Semr{t}, t=t'^l,$$ we have a trivial family $X \times \SP \CC\Semr{t'}$ together with an action of a cyclic group $G$ of order $l$.

Since the family $\mcX \to \SP \CC\Semr{t}$ is projective, there is a relative very ample line bundle $\mathcal{L}$ on $\mcX$. Thus there is a $G$-invariant line bundle $L$ on $X \times \SP \CC\Semr{t'}$. The line bundle $L$ is the pull-back of a line bundle $L_0$ on $X$ via the first projection. Choose a $G$-linearization on $L$. Then the group $G$ acts on the space of sections $H^0(X \times \SP \CC\Semr{t'}, L)=H^0(X, L_0) \otimes \CC\Semr{t'}$. This action naturally extends to an action of $H^0(X, L_0) \otimes \CC\Sem{t'}$. Thus there is a extension of the $G$ action to $X \times \SP \CC\Sem{t'}$. And there is a natural action of $G$ on $X$ (by restricting the action to the closed central fiber) so that the family $\mcX \to \SP \CC\Semr{t}$ is isomorphic to the quotient $(X \times \SP \CC\Semr{t'})/G$.
\end{proof}

And now it comes to the starting point of this paper:
\begin{thm}\label{gob}
Over complex numbers, Question \ref{obs} follows from Conjecture \ref{HT} so geometrically obstructs the latter.
\end{thm}
\begin{proof}
Let $X$ be a smooth projective complex rationally connected variety with a $G$ action, where $G\cong \mathbb{Z}/l\mathbb{Z}$, and $C\cong \PP^1$. Let $G$ act on $C \cong \mathbb{P}^1$ by $z \mapsto \zeta \cdot z$ where $\zeta$ is a primitive $l$-th root of unity. Take the diagonal action of $G$ on $X \times C$ and form the quotient $q:X\times C\rightarrow \mcX=X \times C /G$ with $B=C/G \cong \p$. Then we have projection $\pi: \mcX \rightarrow C/G$ and diagram$$
\xymatrixcolsep{3pc}\xymatrix{
X\times C \ar[r] \ar[d] &\mcX \ar[d]\\
C \ar[r]& B}
$$

If Conjecture \ref{HT} holds for the family $\pi: \mcX \rightarrow B$ with rationally connected general fibers, one can choose two fixed points $x, y$ in $X$ and find a section $s:\p \rightarrow \mcX$ which satisfies $s(0)=q(x,0)$ and $s(\infty)=q(y,\infty)$. Now $s$ gives a $G$-equivariant section $\tilde{s}$ of $\pi_2: X\times C \rightarrow C$ such that $\tilde{s}(0)=(x,0), \tilde{s}(\infty)=(y,\infty)$. Then the projection onto $X$ gives a $G$-equivariant rational curve connecting $x$ to $y$. This leads to Question \ref{obs}.
\end{proof}

\subsection{Twisted curves and Twisted Stable Maps }
In this subsection, we give a short introduction to the theory of twisted curves and $n$-pointed twisted stable maps. We refer to \cite{twisted} and \cite{ol} for more details.

\begin{defn}A \emph{twisted nodal $n$-pointed curve} over $S$ is a diagram
$$\xymatrixcolsep{2.5pc}\xymatrix{
\Sigma_i^{\mcC} \ar@{}[r]|-*[@]{\subset} \ar[rd]  & \mcC \ar[d] \\
 & C \ar[d] \\ &S}$$

where:
\begin{itemize}
 \item $\mcC$ is a tame Deligne-Mumford stack, proper and of finite presentation over
$S$, and \'etale locally is a nodal curve over $S$;
\item $\Sigma^{\mcC}_i \subset \mcC$ are disjoint closed substacks in the smooth locus of $\mcC \to S$;
\item $\Sigma^{\mcC}_i \to S$ are \'etale gerbes;
\item The morphism $\mcC \to C$ exhibits $C$ as the coarse moduli scheme of $\mcC$;
\item $\mcC \to C$ is an isomorphism over $C_{gen}$.
\end{itemize}
\end{defn}

And we can define twisted stable maps (for full categorical definition see \cite{twisted}). We consider a
proper tame Deligne-Mumford stack $\mcM$ admitting a projective coarse
moduli scheme $\mathbf{M}$. We fix an ample invertible sheaf on $\mathbf{M}$.

\begin{defn}\label{Def:twisted-stable-map} A  {\em Twisted Stable
$n$-pointed map of genus $g$ and degree $d$ over $S$}
$$(\mcC \to S, \Sigma_i^{\mcC}\subset \mcC, f\colon  \mcC \to \mcM)$$   consists
of a commutative diagram
$$\xymatrixcolsep{3pc}\xymatrix{
 \mcC \ar[r] \ar[d] & \mcM \ar[d] \\
	 C \ar[r] \ar[d]  & \mathbf{M} \\
S}$$
 along with $n$ closed substacks $\Sigma_i^{\mcC}\subset \mcC$,
satisfying:
\begin{enumerate}
\item $\mcC \to C \to S$ along with $\Sigma_i^{\mcC}$  is a twisted nodal
$n$-pointed curve  over $S$;
\item the morphism $\mcC \to \mcM$ is representable; and
\item $(C\to S, \Sigma_i^C, f\colon C \to \mathbf{M})$ is a stable $n$-pointed map
of degree $d$.
\end{enumerate}
\end{defn}

We recall the main theorem of \cite{twisted}.

\begin{thm}Let $\overline{\mathcal{M}}_{g,n}(\mathcal{M},d)$ be fibered over $\mcS ch/S$, the category of the twisted stable $n$-pointed maps $\mathcal{C} \to \mathcal{M}$ of genus $g$ and degree $d$.
 \begin{enumerate}
 \item The category $\overline{\mathcal{M}}_{g,n}(\mathcal{M},d)$ is a proper algebraic stack.
\item The coarse moduli space $\overline{\mathbf{M}}_{g,n}(\mathcal{M},d)$ of $\overline{\mathcal{M}}_{g,n}(\mathcal{M},d)$ is projective.
 \item There is a commutative diagram
$$\xymatrix{
\overline{\mathcal{M}}_{g,n}(\mathcal{M},d) \ar@{->}[r]\ar[d] & \overline{\mathcal{M}}_{g,n}(\mathbf{M},d) \ar[d] \\
\overline{\mathbf{M}}_{g,n}(\mathcal{M},d)         \ar@{->}[r]     & \overline{\mathbf{M}}_{g,n}(\mathbf{M},d),
 }$$
where the top arrow is proper, quasi-finite and relatively of Deligne-Mumford type, and the bottom
arrow is finite.
\end{enumerate}
\end{thm}
 \begin{rem}\label{stabilize}
 For a morphism between tame Deligne-Mumford stacks $f: \mcX \to \mcY$, we have a natural morphism
$$\xymatrix{
\overline{\mathcal{M}}_{g,n}(\mathcal{X},d) \ar@{->}[r] & \overline{\mathcal{M}}_{g,n}(\mathcal{Y},f_* d)
 }.$$
This map will contract the components which become non-stable after composed with the map $f$-from the definition one can show that all such components are isomorphic to $[\mathbb{P}^1/G]$ with $2$ stacky points $[0]$ and $[\infty]$, where $G$ is some cyclic group with $|G|$ divisible in $k$ and the $G$-action is canonically defined by multiplying a primitive $|G|$-th root of unity.
\end{rem}

\section{Finding $G$-equivariant rational curves}
This section is devoted to prove
\begin{thm}[Theorem \ref{thm:Hypersurface}]
 Let $X$ be a smooth projective separably rationally connected variety over $k$ with $\bar{k}=k$. Assume $G \cong \mathbb{Z}/l \mathbb{Z}$ with $l$ divisible in $k$, and G acts on $X$ and on $\mathbb{P}^1$ by $z \mapsto \zeta z$, where $\zeta$ is a primitive $l$-th root of unity. Then for $x,y$ as $2$-fixed points of the $G$ action on $X$, there is a $G$-equivariant map $f: \mathbb{P}^1 \to X$ with $f(0)=x$ and $f(\infty)=y$.
\end{thm}

Now recalling Theorem \ref{ghs} again: for a rationally connected fibration $\mcX \to B$, by the powerful smoothing of comb argument initiated by Prof.~Koll\'ar, Prof.~Miyaoka and Prof.~Mori in \cite{KMM92RC}, and a specialization argument cancelling monodromy of $C^*\to B$ around multiple fibers, one can find a ``flexible'' curve $C^*$ where the forgetful map $$ {\mcF}_{g,0}: \overline{\mcM}_{g,0}(\mcX, [C^*]) \to \overline{\mcM}_{g,0}(B, \pi_{*}[C^*])$$ is smooth and surjective. Degenerate $(C^* \to B)$ in the Hurwicz scheme $\mcH_{g,B}=\overline{\mcM}_{g,0}(B, \pi_{*}[C^*])$ to contain a component isomorphic to $B$, a preimage of this component will be a section. Quite unexpectedly, we have the following main observation
\begin{Observation} The situation of Question \ref{obs} inherits a variant of the above \emph{Graber-Harris-Starr} argument.
\end{Observation}

By taking Lefschetz pencils, it is easy to have a high genus $G$-equivariant curve $C$ embedded in $X$ connecting $x$ to $y$, an open neighbourhood $\mcU$ of $C$ in the $G$-fixed locus of the Kontsevich's moduli space of stable map $$\overline{\mcM}_{g,2}(X, [C])\{f(p)=x,f(q)=y\}$$ also parametrizes $G$--equivariant embedded maps marked twice connecting $x$ to $y$, and so it maps to ${\mcH}_{g,2}(G)$ as the Huwicz scheme of Galois covers $C\to C/G$ with two specified ramifications. So we have a rational map $$\xymatrixrowsep{1pc}\xymatrix{
{\mcF}^G_{g,2}: \overline{\mcM}^{G}_{g,2}(X, [C])\{f(x_1)=x, f(x_2)=y \} \ar@{.>}[r] & \mcH_{g,2}(G)
}$$ If one can find a good compactification of this rational map, then similar to \cite{GHS03}, the preimage of a good component ($\mathbb{P}^1 {\to} \mathbb{P}^1/G\cong \mathbb{P}^1$ totally ramified at $2$ points)  of some degenerated $G$-cover might give a $G$-equivariant rational curve connecting $x$ to $y$.

\subsection{Moduli compactification and tangent obstruction of ${\mcF}^G_{g,2}$}Instead of working with $G$-equivariant stable maps, passing to the stacky quotient map $[C/G] \to [X/G]$ helps to find such a compactification:~applying Remark \ref{stabilize} to the natural morphism $[X/G] \to \mcB G$ we get the forgetful map between proper algebraic stacks $$\xymatrixcolsep{2pc}\xymatrix{{[\mcF]}^G_{g',2}: \overline{\mcM}_{g',2}([X/G], [C/G])\{f'([x_1])=[x], f'([x_2])=[y] \} \ar[r] & \overline{\mcM}_{g', 2}(\mcB G).}$$
Here $g'$ is the genus of the stacky curve $[C/G]$, in order to apply the lift argument of \emph{Graber-Harris-Starr}'s paper \cite{GHS03}, the remained problem is to analysis the tangent obstruction of this map and to deal with possible monodromy problems which arise in the cover $C \to C/G$. As we will discuss in the following: for the first problem, we restrict to a special type of twisted map where the deformation theory is easier to describe; and for the second we construct a special high genus equivariant curve $C$ admitting a special degeneration that cancels the monodromy.

We restrict to a special type of $G$-equivariant stable map or twisted stable map that has simpler deformation and tangent obstruction for ${[\mcF]}^G_{g',2}$.
\begin{defn}
For a stable nodal curve $C$ with a $G$-action,  containing $p$, $q$ as two marked points, we say $(C,p,q)$ is $G$-simple if the follows are satisfied:
 \begin{itemize}
 \item The $G$ action on $C$ is effective.
 \item $p$, $q$ are fixed points of $G$.
 \item For any node $n$ of $C$, the stabilization subgroup of $G$ that fixed $n$ as $Stab_G(n)$ is trivial.
 \end{itemize}
Let $(C,p,q)$ be a $G$-simple nodal curve, and we assume that there is a $G$-equivariant morphism $f: C \to X$. We say $f: (C, p , q)  \to (X, f(p) , f(q))$ is a $G$-simple map if $f: (C,p,q) \to X$ is a stable map and that $f$ is an immersion.
\end{defn}
\begin{rem}\label{simple}
 We note that a $G$-simple curve can be $G$-equivariantly smoothed to a smooth and irreducible  curve with marking $p$, $q$ which is again $G$-simple. Actually by taking quotients $[C/G]$, they form a special class of \emph{balanced} twisted stable nodal $2$ point curves as defined in \cite{twisted} and \cite{ol}, since  $[C/G]$ has no stacky structure at the nodes. And for a $G$-simple map $f: (C, p , q)  \to X$, clearly it gives a twisted stable map $$[f]: ([C/G], [p/G],[q,G]) \to ([X/G], [f(p)/G] , [f(q)/G]).$$  The normal sheaf $\mcN_{f/X}$ is locally free with a natural $G$-linearilization. In particular, it gives a special type of \emph{balanced} twisted stable $2$ point maps in the sense of \cite{twisted} and \cite{ol}.
\end{rem}
\begin{prop}\label{tangent}
 Let $f: (C, p , q)  \to (X, f(p) , f(q))$ be a $G$-simple equivariant map. Assume $f(p)=x, f(q)=y$, and $H^1(C,f^{*}\mcT_X (-p-q))^G=0$. Then  $$\xymatrixcolsep{2pc}\xymatrix{{[\mcF]}^G_{g',2}: \overline{\mcM}_{g',2}([X/G], [C/G])\{f'([p])=[x], f'([q])=[y] \} \ar[r] & \overline{\mcM}_{g', 2}(\mcB G)}$$ will be smooth and surjective at the point $([f]: [C/G] \to [X/G]).$
\end{prop}
\begin{proof}
We have the following long exact sequence of cohomology groups
\[
0 \to Ext^0(\Omega_C(p+q), \OO_C) \to Ext^0(f^*\Omega_{X}(p+q), \OO_C) \to
\]
\[
\to \mathbb{H}^1({{\RR}Hom_{\OO_C}(\Omega^{\cdot}_f(p+q),\OO_C)}) \to Ext^1(\Omega_C(p+q), \OO_C) \to
\]
\[
\to Ext^1(f^*\Omega_{X}(p+q), \OO_C) \to \mathbb{H}^2({{\RR}Hom_{\OO_C}(\Omega^{\cdot}_f(p+q), \OO_C)})\to 0
\]

The deformation and obstruction space of the stable map $(f: (C, p, q) \to (X, x, y))$ are
\[
Def(f)=\mathbb{H}^1(C,\mathbb{R}Hom_{\OO_C}(\Omega_f^{\cdot}(p+q),\OO_C))
\]
 \[
Obs(f)=\mathbb{H}^2(C,\mathbb{R}Hom_{\OO_C}(\Omega_f^{\cdot}(p+q),\OO_C))
\]
where
$\Omega^{\cdot}_f(p+q)$ is the complex
\[
\begin{CD}
-1 && 0\\
f^*\Omega_X(p+q) @> df^\dagger >> \Omega_C(p+q).
\end{CD}
\]
Since $(f: (C, p, q) \to (X, x, y))$  is $G$-simple, by Remark \ref{simple}, $\Omega^{\cdot}_f(p+q)$ is qusi-isomorphic to $Hom_{\mcO_C}(\Omega_C, \mcN_{f/X})(p+q)$ which is locally free with a natural $G$-linearilization, and the deformation and obstruction space of the stable map $$(f: ([C/G], [p/G], [q/G]) \to ([X/G], [x/G], [y/G]))$$ is the $G$-invariant part of $Def(f), Obs(f)$. Since $C$ admits $G$-equivariant smoothings, the deformation space of the stacky curve $([C/G], [p/G], [q/G])$) will be the $G$-invariant part of $Ext^1(\Omega_C(p+q))$. And the tangent map between $$\overline{\mcM}_{g', 2}([X/G], [C/G])\{[f]([p])=[x], [f]([q])=[y]\}$$ and $\overline{\mcM}_{g', 2}(\mcB G)$ is simply $$\mathbb{H}^1({{\RR}Hom_{\OO_C}(\Omega^{\cdot}_f(p+q),\OO_C)})^G \to Ext^1(\Omega_C(p+q), \OO_C)^G$$

So if $H^1(C, f^*T_X(-p-q))^G=0$, the morphism ${[\mcF]}^G_{g',2}$ is smooth at the point represented by $[f]: ([C/G], [p/G], [q/G]) \to ([X/G], [x/G], [y/G])$ and the forgetful map is smooth at this point. Thus it is surjective when restricted to the unique irreducible component containing this stable map since the forgetful map is also proper.

\end{proof}
\subsection{A pencil in $\overline{\mcM}_{g',2}(\mcB G)$}
Even with Proposition \ref{tangent} in hand, we still meet with delicate monodromy problem in order to find a twisted curve containing $\mathbb{P}^1 {\to} \mathbb{P}^1/G\cong \mathbb{P}^1$ totally ramified at $2$ points in $\overline{\mcM}_{g', 2}(\mcB G)$.

As in Proposition \ref{tangent}, one would like to degenerate the nodal curve $C$ so that the two points $p,q$ mapped to the fixed points come together and lie in an irreducible component which is isomorphic to $\mathbb{P}^1$ with the canonical  $G$ action. It is always possible to degenerate the curve
with the points coming together. But in order that the two points lie in a $\mathbb{P}^1$ with
a $G$ action, the monodromy around $p,q$ has to be inverse to each other. The following proposition gives such a kind of degeneration  in $\overline{\mcM}_{g',2}(\mcB G)$.

\begin{prop}\label{pencil}
There is a $G$-simple nodal curve $\hat{C}$ marked twice at fixed points $p,q$, and a pencil of $G$-simple curves $\mcC \to \mathbb{P}^1$ which deforms $(\mcC_0 \cong \hat{C} , p, q)$  to smooth $G$-simple curve $(\mcC_t, p_t, q_t)$ for general $ t \in \mathbb{P}^1$. And $\mcC_{\infty} = C_1 \cup C_2$ where both $C_1$ and $C_2$ are isomorphic to $\mathbb{P}^1$ with the canonical $G$-action, and the marking $p_{\infty}$ and $q_{\infty}$ of $\mcC_{\infty}$ is concentrated on the $0$ and $\infty$ in $C_1$.

In particular, by taking quotients $[\mcC/G] \to \mathbb{P}^1$ gives a pencil of twisted maps in $\overline{\mcM}_{g',2}(\mcB G)$, which deforms a high genus twisted curve marked twice to contain a component isomorphic to $([\mathbb{P}^1/G],[0/G],[\infty/G])$.
\end{prop}

\begin{proof}
Let $G$ acts on $\PP^1 \times \PP^1$ by $g \cdot ([X_0, X_1], [Y_0, Y_1]) \mapsto ([X_0, \zeta \cdot X_1], [Y_0, Y_1])$, where $\zeta$ is a primitive $l$-th root of unity. Consider the reducible nodal curve
\[
\hat{C}=V((X_0^{2l}-X_1^{2l})Y_0Y_1)\subset \PP^1 \times \PP^1.
\]
There is a natural $G$ action on $\hat{C}$ by restricting the action on $\PP^1 \times \PP^1$. We take two marked points $p=([1, 0], [1, 0]), q=([0, 1], [0, 1])$.

One can directly write down such a pencil. In $\PP^1 \times \PP^1$, take the pencil spanned by the curves $(X_0^{2l}-X_1^{2l})Y_0Y_1=0$ and $(X_0^lY_1+X_1^l Y_0)(X_1^l Y_1+X_0^l Y_0)=0$. Note that the curve defined by $(X_0^lY_1+X_1^l Y_0)(X_1^l Y_1+X_0^l Y_0)=0$ is the union of two smooth rational curve meeting transversely at $2l$ points, both of which has a natural $G$ action. Moreover, the irreducible component $X_0^lY_1+X_1^l Y_0=0$ contains $p, q$ as the fixed points of the $G$ action. A general member of the pencil is a stable curve since $C$ is stable. And every member of the family has a natural $G$ action and contains $p, q$ as two fixed points. So set $p_t=p$, $q_t=q$, this gives a family of $G$-simple curves satisfying out hypothesis.

\end{proof}
\subsection{Final proof of Theorem \ref{thm:Hypersurface}}

With Proposition \ref{tangent} and Proposition \ref{pencil} at hand, we need one more lemma to prove Theorem \ref{thm:Hypersurface}, which is an application of the equivariant smoothing of comb-Theorem \ref{equivariant smoothing}.
\begin{lem}\label{lem:EquivSmoothing}
Assume $G \cong \mathbb{Z}/l\mathbb{Z}$. Let $X$ be a smooth quasi-projective separably rationally connected $k$-variety with a $G$ action. Let $G$ actions on $\PP^1$ by $z \mapsto \zeta z$, where $\zeta$ is a primitive $l$-th root of unity.
\begin{enumerate}
\item Let $f: \PP^1 \rightarrow X$ be a $G$-equivariant map. Then there is a $G$-equivariant very free curve $\tilde{f}:\p \rightarrow X$ with $\tilde{f}(0)=f(0)$, $\tilde{f}(\infty)=f(\infty)$.

\item Let $f_i: C_i \rightarrow X, 1\leq i \leq n$ be a chain of equivariant maps, i.e.$f_i(\infty)=f_{i+1}(0)$ and $C_i \cong \mathbb{P}^1$ for all $i$'s. Then there is a $G$-equivariant map $\tilde{f}: \p \rightarrow X$ with $\tilde{f}(0)=f_1(0)$ and $\tilde{f}(\infty)=f_n(\infty)$.
\end{enumerate}
\end{lem}
\begin{proof}
For part $(1)$, Let $G$ acts on $\mathbb{P}^M$ by $[x_0,...,x_M]\to [x_0,..., \zeta x_M]$ (or any other effective action) with an embedded $G$-equivariant rational curve $f':\mathbb{P}^1 \to \mathbb{P}^M$, e.g.~$x_1=...=x_{M-1}=0$. For $M\gg 0$, we may assume that the equivariant map $f$ is an embedding and $dim\ X \geq 3$ by replacing $X$ with $X \times \mathbb{P}^M$ and replacing $f: \PP^1 \rightarrow X$ with the diagonal $G$-equivariant map $$\hat{f}: \PP^1 \rightarrow X\times \mathbb{P}^M$$ defined by $\hat{f}(x)=(f(x),f'(x)).$ We note there that for $dim\ X\geq 2$, $M=1$ is enough.  Let $C_0$ be the image
of the morphism $f$. Then $(1)$ follows by applying Theorem \ref{equivariant smoothing} to $f_0=f$, $d=2$ and $M=\OO_{C_0}(-0-\infty)$.

For part $(2)$, we may assume that all the $f_i$'s are very free by the first part. Let $f$ be the $G$-equivariant map obtained by gluing the $f_i$'s. Let $(T, o)$ be a pointed smooth curve with a trivial $G$-action. And let $\tilde{\Sigma}$ be $\PP^1\times T$ with the natural diagonal action. There are two $G$-equivariant sections, $s_0=0 \times T, s_\infty=\infty \times T$. Blow up $s_\infty(o)$ with an extension of the $G$-action and equate $s_0,s_{\infty}$ with their strict transforms. Doing this $n$ times, we get a smooth surface $\Sigma$ with a chain of $n$ rational curves as the fiber over $o \in T$. Let $h_0: s_0 \rightarrow X\times T$ and $h_\infty: s_\infty \rightarrow X \times T$ be $T$-morphisms such that $h_0(s_0)=f_1(0)\times T$ and $h_\infty(s_\infty)=f_n(\infty) \times T$. Consider the map $$\mu: Hom_T(\Sigma, X\times T, h_0, h_\infty) \rightarrow T$$ where $\mcH om_T(\Sigma, X\times T, h_0, h_\infty)$ parameterizes $T$-morphisms from $\Sigma$ to $X \times T$ fixing $h_0$ and $h_\infty$, it has a natural $G$ action and the map $\mu$ is $G$-equivariant. Now it is easy to show that $\mu$ is smooth at $f$ since $f_i$'s are very free. So there is a $G$-equivariant smoothing of the morphism $f$ by Corollary \ref{sm-equiv-lifting}.
\end{proof}

\begin{proof}[Proof of Theorem \ref{thm:Hypersurface}]
 Given two fixed points $x$ and $y$, there is a very free rational curve $f_1: C_1\cong \PP^1 \to X$ such that $f(0)=x, f(\infty)=y$. The $G$-orbit of the morphism $f_1$ consists of $l$ very free curves $f_i: C_i \cong \PP^1 \to X, f_i(0)=x, f_i(\infty)=y, i=1, \ldots, l$. Take another $G$-orbit of very free curves $g_i: D_i \cong \PP^1 \to X, g_i(0)=x, g_i(\infty)=y, i=1, \ldots, l$.

Now consider the special $G$-nodal curve as in Proposition \ref{pencil}
\[
\hat{C}=V((X_0^{2l}-X_1^{2l})Y_0Y_1)\subset \PP^1 \times \PP^1.
\]

One can define a $G$-equivariant morphism $f: \hat{C} \to X$ whose restriction to the curve $Y_0=0$ (resp. $Y_1=0$) is the constant map to $x$ (resp. $y$), to the curve $V(X_0-\zeta^i X_1=0) \cong C_i$ the map $f_i$, to the curve $V(X_0-\zeta^i \xi X_1=0) \cong D_i$ the map $g_i$, where $\xi$ is a root of the equation $T^l+1=0$. Take the marking $p=([1, 0], [1, 0]), q=([0, 1], [0, 1])$ as in Proposition \ref{pencil}. The triple $(f: \hat{C} \to X, f(p)=x, f(q)=y)$ is a $G$-equivariant stable map of genus $g$ with two marked points.

We note there that although $\hat{C}$ is a $G$-simple nodal curve, $(f: \hat{C} \to X, f(p)=x, f(q)=y)$ is not necessarily $G$-simple. But we claim that there is a smooth projective separably rationally connected variety $\hat{X}$ with a $G$ action, $2$ fixed points $\hat{x},\hat{y}$, and a $G$-equivariant morphism $X' \to X$ which maps $\hat{x}$,$\hat{y}$ to $x$,$y$ respectively, such that there is a lifting $$\xymatrix{
(\hat{C},p,q) \ar@{->}[r]\ar[rd] & (\hat{X},\hat{x},\hat{y}) \ar[d]\\
    & (X,x,y)
 }$$
 where  $(\hat{f}: \hat{C} \to \hat{X}, \hat{f}(p)=\hat{x}, \hat{f}(q)=\hat{y})$ is $G$-simple. To see this first we take the product $X\times \mathbb{P}^M$ for $M\gg 0$ and take suitable preimages $x'$, $y'$ of $x$, $y$. Here $\mathbb{P}^M$ has a suitable $G$ action similar to the  proof of part $1$ in Lemma \ref{lem:EquivSmoothing}. So we can assume that the very free curves $C_i$'s are all immersed rational curves which are distinct with eath other and has different tangent directions at $x'$ and $y'$. Then we blow-up along the preimages $x'$, $y'$ of $x,y$ on $X\times \mathbb{P}^M$ to get $\hat{X}$ as desired.

 Replacing $X$ by $\hat{X}$, we may assume there is a $G$-simple map $(f: \hat{C} \to X, f(p)=x, f(q)=y)$ which also gives rise to a twisted stable map $[f]: ([\hat{C}/G], [p/G], [q/G]) \to ([X/G], [x/G], [y/G])$. By Theorem \ref{relative smoothing} and part $1$ of Lemma \ref{lem:EquivSmoothing}, up to adding very free curves and smoothing, we may assume that $H^1(\hat{C}, f^*\mcT_X(-p-q))^G=0$. Then by Proposition \ref{tangent}, there is a surjection $$\mcS: \mcU_{([f]: [\hat{C}/G] \to [X/G])} \to \mcV_{\hat{C}}$$ where  $\mcU_{([f]: [\hat{C}/G] \to [X/G])}$ and $\mcV_{\hat{C}}$ are the unique components in the moduli stacks $\overline{\mcM}_{g',2}([X/G], [\hat{C}/G])\{f'([x_1])=[x], f'([x_2])=[y] \}$ and $\overline{\mcM}_{g', 2}(\mcB G)$ containing $[f]$ and $\hat{C}$ respectively.

Now as in Proposition \ref{pencil}, one can deform the curve $[\hat{C}/G]$ in such a way that the two stacky points lie in a stacky $[\PP^1/G]$, where the $G$ action on $\PP^1$ is $[X_0, X_1] \mapsto [X_0, \zeta X_1]$-since the general member $[C_t/G]$ this deformation is parametrized by an irreducible curve (actually $\mathbb{P}^1$), the resulting curve $C_t$ is also in the component $\mcV_{\hat{C}}$ containing $\hat{C}$.

 Thus by surjectivity of $\mcS$ there is a preimage as a twisted nodal curve $\mcC^p_t$ mapped to $[C_t/G]$ with a representable morphism which maps the two marked points to $[x/G]$ and $[y/G]$ in $[X/G]$. By Remark \ref{stabilize}, $$\mcC^p_t \mapsto [C_t/G]$$ will contract non-stable components which are again all isomorphic to $[\mathbb{P}^1/G']$ marked at the two total ramifications where $G'$ is a cyclic group with order divisible in $k$. By replacing with a suitable cover, one can assume that $G'=G$, so this is a chain of stacky curves $[\PP^1/G]$ connected at the total ramification points, which is equivalent to a chain of $G$-equivariant morphisms for $\mathbb{P}^1$ to $X$ connecting $x$ to $y$. So finally by part $2$ of Lemma \ref{lem:EquivSmoothing}, we have a $G$-equivariant morphism from $\PP^1$ to $X$ mapping $0$ and $\infty$ to $x, y$.
\end{proof}

\section{Proof of Theorem \ref{m} using relative G-equivariant smoothing}

By Lemma \ref{lem:basechange}, for an isotrivial family, at least in the formal neighborhood, we can find a ramified base change and a birational modification so that the new central fiber becomes smooth and the Galois group acts on the total space of the formal neighborhood, namely, it satisfies Hypothesis \ref{hyp}.

The first goal is to show that we can make the cyclic base change globally on the curve $B$. Then to get back to the original family, one just need to remember the Galois group action and do things in a $G$-equivariant way.

Given finitely many points $x_1, x_2, \ldots, x_n$ in $B$, and any positive integer $l$, there is a cyclic cover of degree $l$ of $B$ which is totally ramified over $x_1, \ldots, x_n$ (and other points). To see this, take a general Lefschetz pencil which maps $x_1, \ldots, x_n$ (and other points) to $0 \in \PP^1$ and is unramified over these points. Take a degree $l$ map $B_1=\PP^1 \to \PP^1, [X_0, X_1] \mapsto [X_0^l, X_1^l]$ and let $C=B\times_{\PP^1} B_1$ be the fiber product. Then $C$ is the desired cyclic cover. Note that we have the freedom to increase the number of branched points so that the genus of $C$ can be arbitrarily large. We could also choose the cover $C \to B$ so that the preimages of $b_1, \ldots, b_k$ are $l$ distinct points.

For different points on the base $B$, the base change we need may have different degrees. But we can approximate the formal sections one by one (again using the iterated blow-up to fix jet data) so that each time we only need to deal with a single base change. So Theorem \ref{m} is reduced to

\begin{thm}\label{thm:G-equiv}
Let $G$ be a cyclic group of order $l$ and let $\mcX$ (resp. $C$) be a smooth proper variety (resp. a smooth projective curve) with a $G$-action. Let $\pi: \mcX \to C$ be a flat family of rationally connected varieties. Assume the following:
\begin{enumerate}
\item The morphism $\pi$ is $G$-equivariant.
\item There is a $G$-equivariant section $s: C \to \mcX$.
\item The $G$-action on $C$ has a fixed point $p$ and the action of $G$ near $p$ is given by $t \mapsto \zeta t$, where $t$ is a local parameter and $\zeta$ is a primitive $l$-th root of unity.

\item The fiber of $\pi: \mcX \to C$ over the point $p$ is smooth.
\end{enumerate}
Then for any positive integer $N$, and any $G$-equivariant formal section $\whts: \SP \whtOO_{p, C} \to \mcX$, there is a $G$-equivariant section $s'$ which agrees with the formal section $\whts$ to order $N$.
\end{thm}

\subsection{Idea and formal set-up}
The idea of the proof goes back to \cite{HT06}. Namely, we would like to add suitable rational curves to the given section and make $G$-equivariant deformations to produce a new section with prescribed jet data. The only subtlety in the proof is that in general we cannot choose the rational curves to be immersed. So instead of working with the normal sheaf as is done in \cite{HT06}, we work with the complex $\Omega_f$ defined as

\[
\begin{CD}
-1 & & 0 \\
f^*\Omega_X @> df^\dagger >> \Omega_C.
\end{CD}
\]
and its derived dual in the derived category. All the tensor products, duals, pull-backs, and push-forwards in the proof should also be taken as the derived functors in the derived category.

The following is a general form of the commonly used short exact sequences (of normal sheaves) which govern the deformation of a stable map from a nodal domain.

\begin{lem}\label{lem:def}
Let $f:C \cup D \to X$ be a morphism from a nodal curve $C\cup D$ with a single node to a smooth variety $X$ and $f_0$ (resp. $f_1$) the restriction of $f$ to $C$ (resp. $D$). Then
\begin{enumerate}
\item
We have the following distinguished triangles:
\[
\Omega_{f}^{\vee}\otimes \OO_D(-n) \to \Omega_f^{\vee} \to \Omega_f^{\vee} \otimes \OO_C \to \Omega_{f}^{\vee}\otimes \OO_D(-n)[1]
\]
\[
\Omega_{f_0}^{\vee} \to \Omega_f^{\vee}\otimes \OO_{C}  \to \epsilon[-1] \to \Omega_{f_0}^{\vee}[1]
\]
where $n$ is the preimage of the node in $D$, and $\epsilon$ is a skyscraper sheaf supported at the preimage of the node in $C$.
\item Let $G$ be a cyclic group of order $l$. Assume that there is a $G$-action on $C\cup D$ fixing each irreducible component. Then the node is a fixed point of the action and there is a natural $G$-action on all the complexes above. If locally around the node, the action is given by
\[
\begin{CD}
\CC[x, y]/xy @>>> \CC[x, y]/xy \\
(x, y)@>>> (\zeta x, \zeta^{-1} y),
\end{CD}
\]
where $\zeta$ is a primitive $l$-th roots of unity, then the $G$-action on $\epsilon$ is trivial.
\end{enumerate}
\end{lem}
\begin{proof}
The first distinguished triangle comes from restriction to the component $C$.

For the second distinguished triangle, consider the following distinguished triangles and the map between them:
\[
\begin{CD}
\Omega_{C\cup D} \otimes \OO_C @>>> \Omega_f \otimes \OO_C@>>> f^*\Omega_X \otimes \OO_C [1]@>>>\Omega_{C\cup D}\otimes \OO_C[1]\\
@VVV@VVV@|@VVV\\
\Omega_{C}  @>>> \Omega_{f_0} @>>> f_0^*\Omega_X[1] @>>>\Omega_{C}[1]
\end{CD}
\]
Therefore we have distinguished triangles
\[
\Omega_{C\cup D}\otimes \OO_C \to \Omega_C \to Q[1]\to \Omega_{C \cup D}\otimes \OO_C[1]
\]
\begin{equation}\label{eq:tri}
\Omega_f \otimes \OO_C \to \Omega_{f_0} \to Q'[1] \to \Omega_f \otimes \OO_C[1],
\end{equation}
\[
Q[1] \to Q'[1] \to 0 \to Q[2].
\]
where $Q$ is a skyscraper sheaf supported at the node. The last distinguished triangle shows that $Q \cong Q'$. Taking dual of the distinguished triangle (\ref{eq:tri}) gives the second triangle in the lemma.

Part 2 of the lemma can be proved by a local computation. Or we can argue that the sheaf $\epsilon$ corresponds to a $G$-equivariant smoothing of the node. Therefore it has to be $G$-invariant.
\end{proof}

Now we begin the proof.

\begin{proof}[Proof of Theorem \ref{thm:G-equiv}] The proof is divided into two steps.

\subsection{Step 1: Approximation at $0$-th order}

We may assume that $$H^1(C, \mcN_{C/\mcX}(-p))=0$$ by the same argument as in Lemma \ref{lem:EquivSmoothing}.

The section $s$ and the formal section $\whts$ intersect the fiber $\mcX_p$ at two fixed points of the $G$-action. Take a rational curve $D\cong \PP^1$ with a $G$-action as $[X_0, X_1] \mapsto [\zeta X_0, X_1]$, where $\zeta$ is the primitive $l$-th root of unity in the assumptions. By Theorem \ref{thm:Hypersurface}, there is a $G$-equivariant very free curve $ D\cong \PP^1 \to \mcX_p \to \mcX$ which maps $0=[1, 0]$ to $s(p)$ and $\infty=[0, 1]$ to $\whts(p)$. Let $f: C \cup D \to X$ be the nodal curve by combining the section and the curve $D$ and $f_0$ (resp. $f_1$) the restriction of $f$ to $C$ (resp. $D$).

By Lemma \ref{lem:def}, we have the following distinguished triangles:
\[
\Omega_{f}^{\vee}(-\infty) \otimes \OO_C(-p) \to \Omega_f{^{\vee}}(-\infty) \to \Omega_f{^{\vee}}(-\infty) \otimes \OO_D \to \Omega_{f}^{\vee}\otimes \OO_C(-p)[1]
\]
\[
\Omega_{f_0}^{\vee}(-p) \to \Omega_f^{\vee}\otimes \OO_{C}(-p)  \to \epsilon[-1] \to \Omega_{f_0}^{\vee}(-p)[1]
\]
\[
\Omega_{f_1}^{\vee}\otimes \OO_D(-\infty) \to \Omega_f^{\vee}\otimes \OO_{D}(-\infty)  \to \epsilon'[-1]\to \Omega_{f_1}^{\vee}\otimes \OO_D(-\infty)[1]
\]
where $\epsilon$ and $\epsilon'$ are torsion sheaves supported at the node of $C$ and $D$. Every complex has a natural $G$-action, and the $G$-actions on $\epsilon$ and $\epsilon'$ are trivial. Also note that $$\Omega_f(-\infty) \otimes \OO_C \cong \Omega_f \otimes \OO_C.$$

Taking hypercohomology gives long exact sequences
\begin{align}\label{eq:long1}
&0\to \HH^1(\Omega_{f}^{\vee}\otimes \OO_C(-p)) \to \HH^1(\Omega_f^{\vee}(-\infty)) \to \HH^1(\Omega_f^{\vee}(-\infty) \otimes \OO_D)\\
 \to &\HH^2(\Omega_{f}^{\vee}\otimes \OO_C(-p))
\to \HH^2(\Omega_f^{\vee}(-\infty)) \to \HH^2(\Omega_f^{\vee}(-\infty) \otimes \OO_D) \to \ldots,\nonumber
\end{align}

\begin{align}\label{eq:long2}
&0\to \HH^1(\Omega_{f_0}^{\vee}\otimes \OO_C(-p)) \to \HH^1(\Omega_f^{\vee}\otimes \OO_{C}(-p)) \to \epsilon\\
\to & \HH^2(\Omega_{f_0}^{\vee}\otimes \OO_C(-p)) \to \HH^2(\Omega_f^{\vee}\otimes \OO_{C}(-p))  \to 0,\nonumber
\end{align}
and
\begin{align}\label{eq:long3}
&0\to \HH^1(\Omega_{f_1}^{\vee}\otimes \OO_D(-\infty)) \to \HH^1(\Omega_f^{\vee}\otimes \OO_{D}(-\infty))  \to \epsilon'\\
\to & \HH^2(\Omega_{f_1}^{\vee}\otimes \OO_D(-\infty)) \to \HH^2(\Omega_f^{\vee}\otimes \OO_{D}(-\infty))  \to 0.\nonumber
\end{align}
 Note that $\Omega_{f_0}^{\vee}$ is quasi-isomorphic to $\mcN_{C/\mcX}[-1]$. Thus by the second long exact sequence,
\[
\HH^2(\Omega_{f_0}^\vee \otimes \OO_C(-p))=\HH^2(\Omega_f^{\vee} \otimes \OO_C(-p))=0.
\]

Note that $\Omega_{f_1}^{\vee}$ is quasi-isomorphic to a shifted sheaf $\mcN[-1]$, where $\mcN$ is defined as the quotient in
\[
0 \to T_D \to f^*T_X \to \mcN\cong f^*T_X/T_D \to 0.
\]
Since $f^*T_X$ is globally generated, $\HH^2(\Omega_{f_1}^{\vee}\otimes \OO_D(-\infty))=0$. Then by the third long exact sequence,
\[
\mathbb{H}^2(\Omega_{f}^{\vee}\otimes \OO_D(-\infty))=0.
\]

Therefore by the long exact sequence (\ref{eq:long1}),
\[
\HH^2(\Omega_f^{\vee}(-\infty))=0,
\]
and thus the $G$-equivariant deformation of the nodal curve $C\cup D$ with the point $\infty$ fixed is unobstructed.

Then by the long exact sequences (\ref{eq:long1}), (\ref{eq:long3}) and the vanishing, the composition of maps
\[
\HH^1(\Omega_f^{\vee}(-\infty))^G \to \HH^1(\Omega_f^{\vee}(-\infty) \otimes \OO_D)^G \to \epsilon'
\]
is surjective. Thus there is a $G$-equivariant deformation with $\infty$ fixed which smooths the node between $C$ and $D$.

\subsection{Step 2: Approximation at higher order}

Assume that we have a section, still denoted by $s$, which agrees with $\whts$ to the $k (\geq 0)$-th order. We want to find a section agreeing with $\whts$ to order $k+1$.

Now let $\mcX_{k+1}$ be the $(k+1)$-th iterated blow-up of $\mcX$ associated to the formal section $\whts$. Then $G$ also acts on $\mcX_{k+1}$ and the projective to $C$ is $G$-equivariant. By abuse of notations, still denote the strict transforms of $s$ and $\whts$ by $s$ and $\whts$. Then they both intersect the exceptional divisor $E_{k+1} \cong \PP^d$ at fixed points of $G$. Assume the intersection points are different, otherwise there is nothing to prove.

Again we assume that $H^1(C, \mcN_{C/\mcX_{k+1}}(-p))=0$.

The key lemma is the following.

\begin{lem}\label{comb}
There is a comb $f: C\cup D \to \mcX_{k+1}$ from a nodal domain consisting of the given section $s(C)$ and suitable rational curves in the fiber such that
\begin{itemize}
\item $D=D_{k+1} \cup \cup_{j=1}^l R_j$, where $D_{k+1}\cong \PP^1$ and $R_j=\cup_{i=1}^k D_{ij}$ is a chain of rational curves. Denote by $x_j$ the node that connects $D_{k+1}$ to $R_j$.
\item There is a $G$-action on $D$ in the following way. The $G$-action on $D_{k+1}$ is given by
\[
[X_0, X_1] \mapsto [X_0, \zeta^{-1} X_1].
\]
The group $G$ acts on $R_j, j=1, \ldots, l$ via a cyclic permutation among them. In particular, the points $x_j \in D_{k+1}$ are conjugate to each other under the $G$-action.

\item The morphism $f: C \cup D \to X$ is $G$-equivariant.

\item The $G$-fixed point $\infty=[0, 1]$ on $D_{k+1}$ is mapped to $\whts(p)$, and $0=[1, 0]$ on $D_{k+1}$ connects $C$.
\item The morphism $f: C \cup D$ is an immersion except at $0$ and $\infty$ in $D_{k+1}$.
\item The complex $\Omega_f^{\vee}$ satisfies the following vanishing conditions.

\begin{equation}\label{1}
\HH^2(\Omega_f^{\vee} \otimes \OO_C(-p))=\HH^2(\Omega_f^{\vee} \otimes \OO_{D_{k+1}}(-0-\infty))=\HH^2(\Omega_f^{\vee} \otimes \OO_{D_{i j}}(-1))=0,
\end{equation}

\begin{equation}\label{2}
\HH^2(\Omega_f^\vee(-\infty))=0,
\end{equation}
\begin{equation}\label{3}
\HH^2(\Omega_f^\vee\otimes \OO_{D_{k+1}}(-\infty-x_1-\ldots-x_l))^G=0.
\end{equation}

\end{itemize}
\end{lem}

The construction is essentially the same as the one in \cite{HT06}, with the only difference coming from the consideration of the $G$-action. For an illustration of the comb $C\cup D$, see Figure. \ref{fig:comb} below and for the configuration of the comb with respect to the iterated blow-up $\mcX_{k+1}$, see Figure. \ref{fig:con}.

\setlength{\unitlength}{0.8in}
\begin{figure}[h]
 \centering
 \includegraphics[width=4.4in]{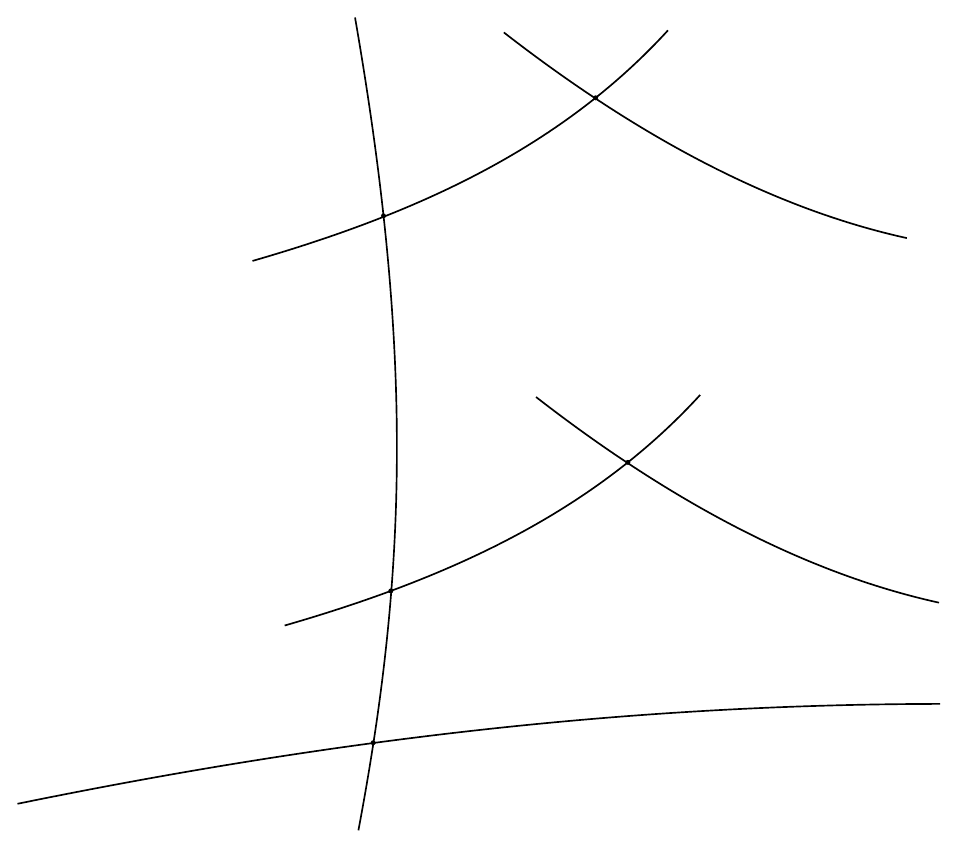}
 \put(-3.55,4.85){$D_{k+1}$}
 \put(-4.25,3.32){$R_1$}
 \put(-3.20,3.50){$x_1$}
 \put(-3.0, 2.50){$\dots\dots$}
 \put(-4.05,1.22){$R_l$}
 \put(-2.92,4.00){$D_{k1}$}
 \put(-3.18,1.30){$x_l$}
 \put(-1.52,4.00){$D_{k-1,1}$}
 \put(-2.72,1.90){$D_{kl}$}
 \put(-1.32,1.90){$D_{k-1,l}$}
 \put(-0.20,3.40){$\dots\dots$}
 \put(-0.05,1.35){$\dots\dots$}
 \put(-0.05,0.55){$C$}
 \caption{The comb $C\cup D$}
 \label{fig:comb}
\end{figure}

\begin{figure}[h]
 \centering
 \includegraphics[width=4.4in]{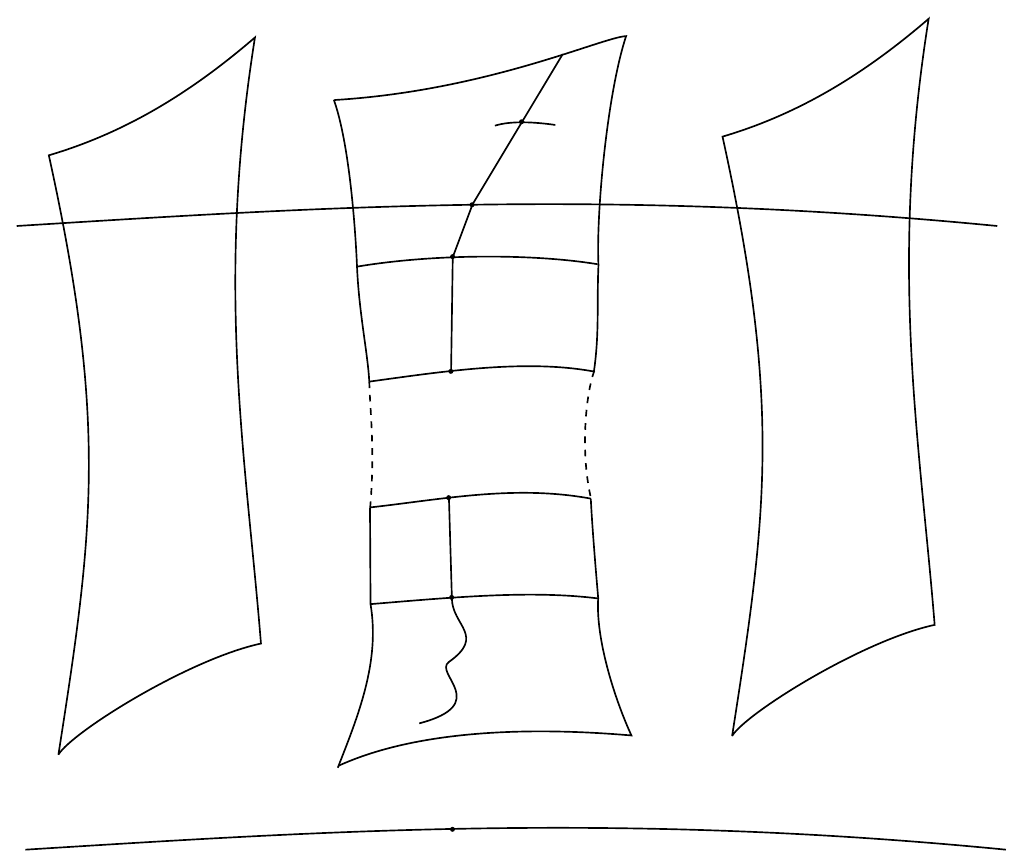}
 \put(-3.25,4.40){$E_{k+1}\cong\mathbb{P}^d$}
 \put(-3.60,3.75){$D_{k+1}\xrightarrow[]{l:1}L$}
 \put(-2.55,4.05){$\wht{s}(p)$}
 \put(-2.95,3.40){$S(p)$}
 \put(-3.25,3.12){$y_k$}
 \put(-3.45,2.70){$y_{k-1}$}
 \put(-3.05,3.10){$D_{k1}\dots D_{kl}$}
 \put(-3.05,2.95){are mapped}
\put(-3.05, 2.80){to this fiber}
 \put(-3.90,2.90){$E_k$}
 \put(-3.05,2.25){$\cdots$}
 \put(-3.25,1.85){$y_1$}
 \put(-3.05,1.83){$D_{11}\dots D_{1l}$}
\put(-3.05,1.68){are mapped}
\put(-3.05, 1.53){to this fiber}
 \put(-3.87,1.65){$E_1$}
 \put(-3.25,1.30){$y_0$}
 \put(-2.97,1.20){$D_{01}$}
 \put(-3.00,1.00){$\cdots$}
 \put(-3.05,0.80){$D_{0l}$}
 \put(-3.87,1.00){$E_0$}
 \put(-3.05,0.00){$P$}
 \put(-0.05,3.40){$S(C)$}
 \put(-0.05,0.00){$C$}
 \caption{Construction of the comb $C\cup D$}
\label{fig:con}
\end{figure}

\begin{proof}[Proof of Lemma \ref{comb}]
The line $L$ in $E_{k+1}\cong \PP^d$ joining $s(p)$ and $\whts(p)$ is invariant and intersects the exceptional divisor $E_{k}$ of $\mcX_{k+1}$ at a unique point $y_k$, which is necessarily a fixed point of $G$. Then there are $3$ fixed points in the line $L$ and thus all points are fixed points of $G$. Take a curve $D_{k+1}\cong \PP^1$. We impose a $G$-action on it by
\[
[X_0, X_1] \mapsto [X_0, \zeta^{-1} X_1].
\]
Take an $l$-to-$1$ $G$-equivariant map from $D_{k+1}$ to the line $L$ such that $0=[1, 0]$ is mapped to $s(p)$ and $\infty=[0, 1]$ is mapped to $\whts(p)$. There are $l$ points $x_1, \ldots, x_l$, which lie in the same orbit of $G$, being mapped to the point $y_k \in E_k \cap E_{k+1}$, where $E_{k}$ and $E_{k+1}$ are exceptional divisors of the $(k+1)$-th iterated blow-up.

The exceptional divisor $E_k$ is isomorphic to the blow-up of $\PP^d$ at a point, thus is a $\PP^1$-bundle over $\PP^{d-1}$. Let $D_{k, 1}, \ldots, D_{k, l}$ be $l$ copies of $\PP^1$ each mapped isomorphically to the fiber curve $\PP^1$ containing the point $y_k$.

Inductively, let $y_{i}$ be the intersection point of $D_{i+1, 1}$ with $E_i$ and $D_{i, 1}, \ldots, D_{i, l}$ be $l$ copies of $\PP^1$ each mapped isomorphically to the fiber $\PP^1$ containing the point $y_i$ for all $i=k-1, \ldots, 1$.

Finally let $y_0$ be the point of the intersection of $D_{1, 1}$ with the strict transform of $\mcX|_p$ and let $D_{0, 1} \ldots, D_{0, l}$ be $l$ copies of $\PP^1$ mapped to a very free curve in the strict transform of $\mcX_p$ intersecting $E_1$ at the point $y_0$. We may also assume that the maps are immersions.

Let $R_j$ be the chain of rational curves $\cup_{i=1}^{k}D_{i, j}$ connected to $D_{k+1}$ at the point $x_j$ for $j=1, \ldots, l$, and let $D$ be the curve $D_{k+1} \cup \cup_{j=1}^{l}R_{ j}$. There is a natural $G$-action on $D$, which permutes the $l$-chains of rational curves $R_j$ and acts on the irreducible component $D_{k+1}$ as specified above.

The restriction of the complex $\Omega_f^{\vee}$ to each curve $D_{i, j}$ is quasi-isomorphic to the normal sheaf with a shift $N_{f}[-1]$ (since the comb is an immersion along such curves). One can compute the restriction of $N_f$ to each curve $D_{i, j}$ as follows (see the proof of Sublemma 27, \cite{HT06}).
\begin{equation}\label{eqN}
\mcN_f|_{D_{i, j}}=\begin{cases}
\OO^{\oplus d}, & 1 \leq i \leq k,\\
\oplus_{n=1}^{d-1} \OO(a_n) \oplus \OO, a_n \geq 1 & i=0.
\end{cases}
\end{equation}

We now compute $\Omega_{f_{k+1}}^{\vee}$ on $D_{k+1}$, where $f_{k+1}$ is the restriction of the map to $D_{k+1}$ (i.e. the degree $l$ multiple cover of the line in $\PP^{d-1}$). This complex is quasi-isomorphic to the complex
\[
\begin{CD}
0 & & 1 \\
T_{D_{k+1}}\cong \OO(2) @>>> f_{k+1}^*T_{\mcX_{k+1}}\cong \OO(2l)\oplus \oplus_{i=1}^{d-2}\OO(l) \oplus \OO(-l),
\end{CD}
\]
Also note that the sheaf map $T_{D_{k+1}}\to f_{k+1}^*T_{\mcX_{k+1}}$ is injective and is the composition of maps $\OO(2) \to \OO(2l) \to f_{k+1}^*T_{\mcX_{k+1}}$.

We have a distinguished triangle
\begin{equation}\label{tricky}
\Omega_{f_{k+1}}^{\vee} \to \Omega_f^{\vee} \otimes \OO_{D_{k+1}} \to \epsilon[-1] \oplus \oplus_{j=1}^{l} \epsilon_j[-1]\to \Omega_{f_{k+1}}^{\vee}[1],
\end{equation}
where $\epsilon$ is a torsion sheaf supported at the node connecting $D_{k+1}$ and $C$, and $\epsilon_j$ is a torsion sheaf supported at the node connecting $D_{k+1}$ and $D_{k,j}$. The group $G$ acts on $\epsilon$ by the trivial action and acts on $\epsilon_j$ by permutation.

So the restriction of $\Omega_f^{\vee}$ to $D_{k+1}$ is quasi-isomorphic to the complex
\[
\begin{CD}
0 & & 1 \\
\OO(2) @>>>\OO(2l)\oplus \oplus_{i=1}^{d-2}\OO(l) \oplus \OO(1).
\end{CD}
\]
Since the above map maps the sheaf $\OO(2)$ injectively into the sheaf $\OO(2l)$, this complex is quasi-isomorphic to the shifted sheaf
\[
Q \oplus \oplus_{i=1}^{d-2}\OO(l) \oplus \OO(1) [-1],
\]
where $Q$ is the torsion sheaf defined as the quotient of $\OO(2) \to \OO(2l)$. Note that the $\OO(1)$ direction is the normal direction of the fiber.

Finally, the restriction of $\Omega_f^{\vee}$ to $C$ fits into the distinguished triangle
\[
\mcN_{C/\mcX_{k+1}}[-1] \to \Omega_f^\vee \otimes \OO_C \to \epsilon_0 \to \mcN_{C/\mcX_{k+1}},
\]
where $\epsilon_0$ is a torsion sheaf supported at the node.

Then the vanishing conditions (\ref{1}) are immediate from the identifications above.

By the distinguished triangle
\[
\Omega_f^\vee \otimes \OO_C(-p) \to \Omega_f^\vee(-\infty) \to \Omega_f^\vee \otimes \OO_D(-\infty) \to \Omega_f^\vee \otimes \OO_C(-p)[1]
\]
and the three vanishing results in \ref{1}, we know that
\[
\HH^2(\Omega_f^\vee(-\infty))=0.
\]
This is the vanishing in (\ref{2}).

The vanishing in (\ref{3}) needs a little bit more work since it is only the $G$-invariant part of the hypercohomology group that vanishes.  First notice the following.
\begin{lem}\label{lem:smoothCD}
Assume only that the comb $C \cup D$ satisfies vanishing results (\ref{1}) and (\ref{2}). Then a general $G$-equivariant deformation of $C\cup D$ with $\infty$ fixed is unobstructed and smooths the node connecting $C$ and $D_{k+1}$.
\end{lem}
\begin{proof}
The vanishing result (\ref{2}) implies that the $G$-equivariant deformation of $C\cup D$ with $\infty$ fixed is unobstructed.

We first consider the following distinguished triangles
\begin{equation}\label{t2}
\Omega_f^\vee \otimes \OO_{D}(-\infty-0) \to \Omega_f^\vee(-\infty) \to \Omega_f^\vee(-\infty)\otimes \OO_C \to \Omega_f^\vee \otimes \OO_{D}(-\infty-0)[1]
\end{equation}
and
\begin{align} \label{t1}
&\oplus_{j=1}^l \Omega_f^\vee  \otimes \OO_{R_j}(-x_j) \to \Omega_f^\vee \otimes \OO_{D}(-\infty-0) \to \Omega_f^\vee \otimes\OO_{D_{k+1}}(-\infty-0) \\
\to &\oplus_{j=1}^l \Omega_f^\vee  \otimes \OO_{R_j}(-x_j)[1].\nonumber
\end{align}

Recall that $\Omega_f^\vee  \otimes \OO_{R_j}$ is quasi-isomorphic to a shifted normal sheaf $\mcN_f \otimes \OO_{R_j}[-1]$, and the sheaves $\mcN\otimes \OO_{R_j}$ are locally free and globally generated by (\ref{eqN}). Therefore
\[
\HH^2(\oplus_{j=1}^l \Omega_f^\vee  \otimes \OO_{R_j}(-x_j))=0,
\]
and thus by the distinguished triangle (\ref{t1}),
\[
\HH^2(\Omega_f^\vee \otimes \OO_{D}(-\infty-0))=0,
\]
which, combined with the long exact sequence of hypercohomology of the distinguished triangle (\ref{t2}), implies that the map
\begin{equation}\label{4}
\HH^1(\Omega_f^\vee(-\infty))^G \to \HH^1(\Omega_f^\vee(-\infty)\otimes \OO_C)^G
\end{equation}
is surjective.

Then we look at the distinguished triangle
\[
\Omega_{f_0}^\vee \to \Omega_f^\vee \otimes \OO_C \to \epsilon_0[-1] \to \Omega_{f_0}^\vee[1],
\]
where $f_0$ is the restriction of $f$ to $C$ and $\epsilon$ is a skyscraper sheaf supported at the point $p$.

By the vanishing results (\ref{1}), the map
\begin{equation}\label{6}
\HH^1(\Omega_f^\vee(-\infty) \otimes \OO_C)^G \to (\epsilon_0)^G=\epsilon_0
\end{equation}
is surjective.

Note that $\Omega_f^\vee \otimes \OO_C \cong \Omega_f^\vee(-\infty)\otimes \OO_C$. Combining this identification and the surjectivity of maps in (\ref{6}) and (\ref{4}), we have proved that a general $G$-equivariant deformation with $\infty$ fixed smooths the node connecting $C$ and $D_{k+1}$.
\end{proof}

We have a distinguished triangle
\[
\Omega_{f_{k+1}}^\vee(-\infty) \to \Omega_f^\vee \otimes \OO_{D_{k+1}}(-\infty) \to \epsilon[-1] \oplus \oplus_{j=1}^{l} \epsilon_j[-1]\to \Omega_{f_{k+1}}^\vee(-\infty)[1],
\]
where $\epsilon$ is a torsion sheaf supported at $0 \in D_{k+1}$. This induces a map
\begin{equation}\label{8}
\HH^1(\Omega_f^\vee \otimes \OO_{D_{k+1}}(-\infty))^G \to \epsilon
\end{equation}
By Lemma \ref{lem:smoothCD}, a general deformation of $C\cup D$ with $\infty$ fixed is unobstructed and smooths the node connecting $C$ and $D_{k+1}$ (note that the proof of this result is independent of the vanishing (\ref{3})). Thus the composition
\[
\HH^1(\Omega_f^\vee(-\infty) \to \HH^1(\Omega_f^\vee \otimes \OO_{D_{k+1}}(-\infty))^G \to \epsilon
\]
is surjective. So the map in (\ref{8}) is also surjective.

Recall that $\Omega_f^\vee \otimes \OO_{D_{k+1}}(-\infty)$ is quasi-isomorphic to the shifted sheaf
\[
(Q \oplus \oplus_{i=1}^{d-2}\OO(l) \oplus \OO(1))\otimes \OO_{D_{k+1}}(-\infty) [-1],
\]
and the $\OO(1)$ direction is the normal direction of the fiber.

Moreover the map in (\ref{tricky}) is can be written as
\begin{align*}
&Q \oplus \oplus_{i=1}^{d-2}\OO(l) \oplus \OO(-l) [-1] \to Q \oplus \oplus_{i=1}^{d-2}\OO(l) \oplus \OO(1) [-1]\\
\to  &\epsilon[-1] \oplus \oplus_{j=1}^{l} \epsilon_j[-1]
\to Q \oplus \oplus_{i=1}^{d-2}\OO(l) \oplus \OO(-l)
\end{align*}

Thus only the $\OO(1)\otimes \OO_{D_{k+1}}(-\infty)$
 summand may have a non-zero map to $\epsilon$ in the above evaluation map in (\ref{8}). Thus the unique section in this summand (i.e. the section of $H^0(\OO(1)\otimes \OO_{D_{k+1}}(-\infty))=H^0(\OO_{D_{k+1}})$) is mapped to a non-zero element in $\epsilon$. Furthermore, this unique section, thought of as a section in $$\HH^1(\Omega_f^\vee \otimes \OO_{D_{k+1}})^G$$
via the inclusion
\[
\HH^1(\Omega_f^\vee \otimes \OO_{D_{k+1}}(-\infty))^G \to \HH^1(\Omega_f^\vee \otimes \OO_{D_{k+1}})^G
\]
only vanishes at $\infty \in D_{k+1}$. Therefore the map
\begin{equation}\label{eq:surj}
H^0( \OO(1) \otimes \OO_{D_{k+1}}(-\infty))^G \to (\oplus_{j=1}^l \epsilon_j)^G
\end{equation}
is surjective.

To prove the vanishing in (\ref{3}), we only need to consider the $\OO(1)$ summand since all the other summands have enough positivity to kill the higher cohomology $\HH^2$. For the $\OO(1)$ summand, consider the short exact sequence
\[
0 \to \OO(1)\otimes \OO_{D_{k+1}}(-\infty-x_1-\ldots-x_l) \to \OO(1) \otimes \OO_{D_{k+1}}(-\infty) \to \oplus_{j=1}^l \epsilon_j \to 0,
\]
which induces a map on the $G$-invariant part of cohomology
\begin{align*}
&H^0( \OO(1) \otimes \OO_{D_{k+1}}(-\infty))^G \to (\oplus_{j=1}^l \epsilon_j)^G \\
\to &H^1(\OO(1)\otimes \OO_{D_{k+1}}(-\infty-x_1-\ldots-x_l))^G \to H^1( \OO(1) \otimes \OO_{D_{k+1}}(-\infty))^G.
\end{align*}
Since the map (\ref{eq:surj}) is surjective and $H^1( \OO(1) \otimes \OO_{D_{k+1}}(-\infty))^G$ vanishes, we have
\[
H^1(\OO(1)\otimes \OO_{D_{k+1}}(-\infty-x_1-\ldots-x_l))^G=0,
\]
and thus
\[
\HH^2(\Omega_f^\vee \otimes \OO_{D_{k+1}}(-\infty-x_1-\ldots-x_l))^G=0.
\]
\end{proof}

We now finish the proof of step 2. Consider the distinguished triangles
\[
\Omega_{f}^{\vee}(-\infty) \otimes \OO_C(-p) \to \Omega_f^{\vee}(-\infty) \to \Omega_f^{\vee}(-\infty) \otimes \OO_D \to \Omega_{f}^{\vee}\otimes \OO_C(-p)[1]
\]
\begin{align*}
&\Omega_{f}^\vee\otimes \OO_{D_{k+1}}(-\infty-x_1-\ldots-x_l) \to \Omega_f^\vee \otimes \OO_{D}(-\infty)\\
 \to &\oplus_{j=1}^l \Omega_f^\vee \otimes \OO_{R_j} \to \Omega_{f}^\vee\otimes \OO_{D_{k+1}}(-\infty-x_1-\ldots-x_l)[1].
\end{align*}

The vanishings in (\ref{1}), (\ref{3}) imply that the map
\begin{equation}\label{7}
\HH^1(\Omega_f^\vee(-\infty))^G \to \HH^1(\Omega_f^\vee \otimes \OO_{D}(-\infty))^G \to \HH^1(\oplus_{j=1}^l \Omega_f^\vee \otimes \OO_{R_j})^G
\end{equation}
is surjective (note that $\Omega_f^\vee \otimes \OO_{R_j}\cong \Omega_f^\vee(-\infty) \otimes \OO_{R_j}$).

Since the $G$-action on the chain of rational curves $R_j$ is permutation. There is a section of
\[
\HH^1(\oplus_{j=1}^l \Omega_f^\vee \otimes \OO_{R_j})^G
\]
which is mapped to a non-zero element in the $G$-invariant part of the torsion sheaf supported at the nodes on $R_j, j=1, \ldots, l$ if and only if there is a section of
\[
\HH^1( \Omega_f^\vee \otimes \OO_{R_j})
\]
which is mapped to a non-zero element in the torsion sheaf supported at the nodes on $R_j$ and for some (and hence for all) $j$.

Since the restriction of $\Omega_f^\vee$ to $R_j$ is quasi-isomorphic to $\mcN_f\otimes \OO_{R_j}[-1]$ and $\mcN_f\otimes \OO_{R_j}$ is locally free and globally generated by the vanishing (\ref{1}) or (\ref{eqN}), this follows from the same argument as in \cite{HT06} (in particular, the bottom of P. 187 and P. 188).

So combining this observation with the surjectivity of the map in (\ref{7}) and Lemma \ref{lem:smoothCD}, we have proved that a general $G$-equivariant deformation with $\infty$ fixed smooths all the nodes and produces a new section which agrees with $\whts$ to order $k+1$.
\end{proof}

\section{Appendix: A conceptual proof of Theorem \ref{m}}

In this section, using the tool of pseudo-ideal sheaves, which is invented by Professor Jason Starr and Michael Roth in \cite{RothStarr09}, and observed by Professor Dan Abramovich and Professor Chenyang Xu to be equivalent to \emph{differential graded subscheme} of aplitute $(1,0)$ (\cite{df}), we are going for a conceptual proof of Theorem \ref{m}.

\begin{defn} \label{defn-pis}
Let $Y$ be an algebraic space and $f: X \rightarrow Y$ be a flat, locally finitely presented, proper algebraic stack over $Y$. For every morphism of algebraic spaces $g:Y'\rightarrow Y$, a flat
family of \emph{pseudo-ideal sheaves of $X/Y$ over $Y'$} is a pair
$(\mc{F},u)$ consisting of
\begin{enumerate}
\item[(i)]
a $Y'$-flat, locally finitely presented, quasi-coherent
$\OO_{X_{Y'}}$-module $\mc{F}$, and
\item[(ii)]
an $\OO_{X_{Y'}}$-homomorphism $u:\mc{F} \rightarrow \OO_{X_{Y'}}$
\end{enumerate}
such that the following induced morphism is zero
$$
u':\bigwedge^2 \mc{F} \rightarrow \mc{F}, \ \ f_1\wedge f_2 \mapsto
u(f_1)f_2 - u(f_2)f_1.
$$
\end{defn}

Let $D$ be an effective Cartier divisor in $X$,
considered as a closed subscheme of $X$, and assume $D$ is flat over
$Y$.  Denote by $\mc{I}_D$ the pullback
$$
\mc{I}_D :=
\mc{I}\otimes_{\OO_X} \OO_D
$$ on $\text{Hilb}_{X/Y}\times_X D$.  And denote by
$$
u_D:\mc{I}_D \rightarrow \OO_{\text{Hilb}_{X/Y}\times_X D}
$$
the restriction of $u$. Then we have

\begin{prop} \cite{RothStarr09}\label{prop-flat}
The locally finitely presented, quasi-coherent sheaf $\mc{I}_D$ is
flat over $\text{Hilb}_{X/Y}$.  Thus the pair $(\mc{I}_D,u_D)$ is a
flat family of pseudo-ideal sheaves of $D/Y$ over $\text{Hilb}_{X/Y}$.
\end{prop}

Denote by
$$\xymatrixcolsep{2pc}\xymatrix{
\iota_D:\mathcal{H}ilb_{X/Y} \ar[r]& {\mcP}seudo_{D/Y}}$$

the $1$-morphism associated to the flat family $(\mc{I}_D,u_D)$
of pseudo-ideal sheaves of $D/Y$ over $\text{Hilb}_{X/Y}$.
This is the \emph{divisor restriction map}.

We have the following theorem, due to M. Roth and J. Starr.

\begin{thm}{\cite{RothStarr09}} \label{thm-jason}
Let $X$ be a Deligne-Mumford stack over $\CC$ and let $C_\kappa \subset X$ be a regularly immersed proper substack of $X$. If both
\[
H^1(C_\kappa,\OO_X(-D)\cdot
\textit{Hom}_{\OO_{C_\kappa}}(\mc{I}_\kappa/\mc{I}_\kappa^2,\OO_{C_\kappa}))
\]
 and
\[
H^1(C_\kappa, \textit{Tor}_{\OO_X}(\OO_{C_\kappa}, \OO_D)\cdot
\textit{Hom}_{\OO_{C_\kappa}}(\mc{I}_\kappa/\mc{I}_\kappa^2,\OO_{C_\kappa}))
\]
equal $0$, then the divisor restriction map $\iota_D$ is smooth at $[C_\kappa]$.
\end{thm}

\begin{rem}When $C_\kappa$ is locally complete intersection, and as if one looks at the derived intersection product $$\xymatrixcolsep{2pc}\xymatrix{ C_\kappa \ar@{|->}[r] & C_\kappa \times_{X}^{L} D}$$ which has the well-behaved tangent obstruction $\mcR \Gamma(C,\mcN_{C/\mcX}\otimes \OO_{\mcX_b})$ (\cite{L04}, \cite{L12}). Then assuming Proposition \ref{prop-flat}, Theorem \ref{thm-jason} will follow directly from the statement that $\iota_D$ is smooth at $[C_\kappa]$ if $\mcR \Gamma(C,\mcN_{C/\mcX}\otimes \OO_{\mcX_b})$ has vanishing first term-this is a standard result due to the theory of tangent complex. To see this, we have the short exact sequence $$0\to \textit{Tor}_{\OO_X}(\OO_{C_\kappa}, \OO_D)\cdot
\textit{Hom}_{\OO_{C_\kappa}}(\mc{I}_\kappa/\mc{I}_\kappa^2,\OO_{C_\kappa}) \to \OO_X(-D)\otimes
\textit{Hom}_{\OO_{C_\kappa}}(\mc{I}_\kappa/\mc{I}_\kappa^2,\OO_{C_\kappa})$$ $$ \to \OO_X(-D)\cdot
\textit{Hom}_{\OO_{C_\kappa}}(\mc{I}_\kappa/\mc{I}_\kappa^2,\OO_{C_\kappa})\to 0 $$
Taking the long exact sequence, then vanishing of the two first cohomology groups in Theorem \ref{thm-jason} will imply that $$H^1(C_\kappa,\OO_X(-D)\otimes
\textit{Hom}_{\OO_{C_\kappa}}(\mc{I}_\kappa/\mc{I}_\kappa^2,\OO_{C_\kappa}))=0$$ and the rest is to remember that the first term of $\mcR \Gamma(C,\mcN_{C/\mcX}\otimes \OO_{\mcX_b})$ is exactly $H^1(C_\kappa,\OO_X(-D)\otimes
\textit{Hom}_{\OO_{C_\kappa}}(\mc{I}_\kappa/\mc{I}_\kappa^2,\OO_{C_\kappa}))$.
\end{rem}

Here is the idea of the proof of Theorem \ref{m} using pseudo-ideal sheaves. We first show that we can deform the restriction of a section $s$ to a formal neighborhood of the singular fiber deforms to another formal section which agrees with the given formal section $\wht{s}$ to any pre-specified order. In particular, this deformation gives a deformation of the corresponding pseudo-ideal sheaves. Then we use Theorem \ref{thm-jason} to show that this deformation of pseudo-ideal sheaves lifts to a deformation of global sections. In order to apply Theorem \ref{thm-jason}, we need to control the cohomology of the curve in the central fiber. To do this, we replace the formal neighborhood of a singular fiber by a smooth Deligne-Mumford stack. This is possible by Proposition \ref{lem:basechange}.

\subsection{A remark on $R$-equivalence}

We firstly recall
\begin{defn}
Let $X$ be a variety over an arbitrary field $k$. Let $x$, $y$ be $2$ points in $X(k)$, they are $R$-equivalent if there is a $k$-morphism $f: \mathbb{P}^1 \to X$ such that $f(0)=x$, $f(\infty)=y$.
\end{defn}

\begin{prop}\label{prop:Requiv}
Let $\mcX \to \SP\CC\Semr{t}$ be an family of rationally connected varieties satisfying Hypothesis \ref{hyp}, esp.~isotrivial family (Proposition \ref{lem:basechange}). Then the $R$-equivalence class of rational points consists of a unique class.
\end{prop}

Note that it is not known in general that the $R$-equivalence class of rational points on a rationally connected variety defined over the Laurent field $\CC\Semr{t}$ is finite.

\begin{proof}
By the assumption of Hypothesis \ref{hyp}, there is a Galois extension of the field $K\subset \wht{K}=K\Semr{t'}/({t'}^l-t)$ with Galois group $G\cong \ZZ/l\ZZ$ such that $\mcX_{\wht{K}}=\mcX_K\otimes_K \SP \wht{K}\cong \mcX'$ where $\mcX'$ is a smooth family with a $G$ action such that all compatibilities are satisfied. In other words, we have the following Cartesian diagram
$$
\xymatrixcolsep{3pc}\xymatrix{
\mcX_{\wht{K}} \ar[r] \ar[d]&\mcX_K \ar[d]\\
\SP\wht{K} \ar[r] &\SP K
}
$$
Let $\whtOO=\CC\Sem{t'}$. The group $G$ acts on $\SP \whtOO$ and the smooth family $\mcX'$ extends to a smooth family $\mcX' \times \SP \whtOO$ which we denote $\mcX'$ again, in such a way that the projection onto the second factor is $G$-equivariant. In particular, $G$ acts on the central fiber $\mcX'_0$ naturally.

Let $s_1, s_2$ be two rational points of $\mcX$. They induce two $\wht{K}$- rational points $\wht{s}_1, \wht{s}_2$ of $\mcX_K\otimes_K \SP \wht{K}$, invariant under the action of the Galois group $G$.

By the valuative criterion of properness, we have two $\whtOO$-points of $\mcX'$, still denoted by $\wht{s}_1, \wht{s}_2$. Let $\wht{s}_1^0, \wht{s}_2^0$ be the intersection points of these two section with the central fiber. Since $\wht{s}_1, \wht{s}_2$ are invariant under the Galois group, the two points $\wht{s}_1^0, \wht{s}_2^0$ are fixed points of the $G$-action on $\mcX'_0$.

By Theorem \ref{thm:Hypersurface} and Lemma \ref{lem:EquivSmoothing}, there exists a $G$-equivariant very free curve $f: (\p,0, \infty) \rightarrow (\mcX'_0, \wht{s}_1^0, \wht{s}_2^0)$. Since the morphism is very free, we have $H^1(\PP^1, f^*T_{\mcX'_0}(-0-\infty))=0$. Thus the map
\[
p: Hom(\PP^1 \times \SP \whtOO, \mcX', f|_{0\times\SP\whtOO}=\wht{s}_1, f|_{\infty \times\SP\whtOO}=\wht{s}_2)) \rightarrow \SP\whtOO
\]
 is smooth at $[f]$. Notice that both spaces have a natural $G$ action such that this map is equivariant and $[f]$ is a fixed point of the action. So by Corollary \ref{sm-equiv-lifting}, there is a $G$-equivariant section
\[
\sigma:\SP\whtOO \rightarrow Hom(\p\times \SP \whtOO, \mcX').
\]
i.e., we have a $G$-equivariant map
\[
f:\p\times \SP\whtOO \rightarrow \mcX'
\]
 such that
\[
f\mid_{ 0\times\SP\whtOO}=\wht{s}_1, f\mid_{\infty\times\SP\whtOO}=\wht{s}_2.
\]
Restricting to the generic fiber gives a Galois invariant $\PP^1$ connecting the two rational points $\wht{s}_1, \wht{s}_2$. Thus the two rational points ${s}_1,{s}_2$ are connected by a rational curve defined over $\CC\Semr{t}$.
\end{proof}

\subsection{Framework of the proof}
It is proved in \cite{HT06} that weak approximation is satisfied if the points are chosen such that the fiber over them are smooth. Their proof is a deformation argument. And the proof given below is also a deformation argument. So using the iterated blow-up construction, it suffices to prove weak approximation at a single place of bad reduction $b \in B$.

Let $\wht{s}$ be a formal section in the formal neighborhood of $b$ we want to approximate. There is a regular section $s$, which we assume to lie in the smooth locus of $\mcX$ by taking a resolution of singularities.

As indicated above, we will work with the \emph{smooth} Deligne-Mumford stack obtained by replacing the formal neighborhood of the singular fiber over $b$ with the smooth Deligne-Mumford stack based on Hypothesis \ref{hyp}.

To be more precise, recall that there is a Galois extension of the field $K\subset \wht{K}=K\Semr{t'}/({t'}^l-t)$ with Galois group $G\cong \ZZ/l\ZZ$ such that $\mcX_{\wht{K}}=\mcX_K\otimes_K \SP \hk\cong \mcX'$. And we have the following Cartesian diagram
$$
\xymatrixcolsep{3pc}\xymatrix{
\mcX_{\wht{K}} \ar[r]\ar[d]& \mcX_K \ar[d]\\
\SP\wht{K} \ar[r] &\SP K
}
$$

together a natural action of $G$ on $\mcX'$ and $\SP \wht{K}$.  Let $\whtOO=\CC\Sem{t'}$. The group $G$ acts on $\SP \whtOO$ and the smooth family $\mcX'$ extends smoothly over $\SP \whtOO$ in such a way that the projection onto the second factor is $G$-equivariant. So $G$ acts on the central fiber $\mcX'_0$.

Denote by $\mathfrak{X}' \to B'$ again the new stacky family obtained by replacing the formal neighborhood of $b$ of the original family $\mcX \to B$ with the quotient stack $[\mcX'/G]$.

 The section $s$ and the formal section $\wht{s}$ give two $\whtOO$-points of $\mcX'$, invariant under the action of the Galois group $G$. We still denote the corresponding invariant section and formal section by $s$ and $\wht{s}$. Furthermore, they induce a section $s'$ of the new family $\mcX' \to B'$ and a formal section $\wht{s}'$ of $[\mcX'/G] \to [\SP \whtOO/G]$.

The weak approximation problem for the two families are equivalent.

The proof of Proposition \ref{prop:Requiv} shows that we have a $G$-equivariant map
\[
f:\p\times \SP\whtOO \rightarrow \mcX'
\]
 such that
\[
f\mid_{ 0\times\SP\whtOO}=s|_{\SP \whtOO}, f\mid_{\infty\times\SP\whtOO}=\wht{s}
\]
We may assume that the morphism is a closed immersion up to replacing the family with product and taking the graph of $f$.

This gives a closed immersion
\[
i:[\PP^1 \times \SP\whtOO/G] \rightarrow [\mcX'/G].
\]

Let $D$ be the divisor in $[\mcX'/G]$ over the closed point of $[\SP \whtOO/G]$ and $E$ its restriction to this ruled surface.

Let $N$ be the order to which we want the regular section to agree with the given formal section. Denote the stacky curve in $D$ by $C_0$. Let $C_s$ be the union of  the subscheme consisting of curve $N \cdot C_0$ and the section $s'(B')$. Let $C_\kappa$ be a comb obtained by attaching very free rational curves in general fibers along general normal directions to $C_s$. Notice that the pseudo-ideal sheaf obtained by restricting $C_\kappa$ to $N\cdot D$ is the same as that obtained by restricting $C_s$, .

 Now we need to use Proposition \ref{prop:Requiv} to produce a pencil of pseudo-ideal sheaves.
\subsection{A non-separated pencil in $\mcP seudo_{\mcX_b}$}
\begin{lem}\label{lem:ps}
There is a family of pseudo-ideal sheaves $\mc{I}_t, t \in \PP^1$ in $N \cdot D$ such that $\mc{I}_0$ is isomorphic to the pseudo-ideal of the restriction of $C_\kappa$ to $N\cdot D$, and a general member of the family is isomorphic to the restriction of $\wht{s}'$ to $N \cdot D$.
\end{lem}
\begin{proof}
It is easy to see that there is a ruled surface
\[
\pi: P=[\PP^1 \times C /G]\rightarrow [C /G]
\]
for some curve $C$ which has a $G$ action (for example, take $C$ to be a cyclic cover of $B$ totally ramified over the point we want to approximate), together with two divisors $0_P$ and $\infty_P$,
 such that
\begin{enumerate}
\item the base change of this ruled surface to $\SP{\whtOO_{0, [C/G]}}$ is isomorphic to
\[
[\PP^1 \times \SP \whtOO /G]\rightarrow [\SP \whtOO /G];
\]
\item the restriction of $0_P$ (resp. $\infty_P$) to ${\SP{\whtOO}}$ is congruent to $C_\kappa|_{\SP{\whtOO}}$ (resp. $\wht{s}'$) modulo $(N+1) \cdot E$.
\end{enumerate}
On $P$ consider the invertible sheaf $\mc{L} = \OO_P(\infty_P - 0_P - \pi^*(N \cdot E))$, where $N$ is the order to which we want the regular section to agree with the given formal section. Since the restriction to the generic fiber of $P\rightarrow [C/G]$ is a degree $0$ invertible sheaf on $[C/G]$, it has a $1$-dimensional space of global sections.  Thus the push-forward of $\mc{L}$ to $[C/G]$ is a torsion-free, coherent $\OO_{[C/G]}$-module of rank $1$, i.e., it is an invertible sheaf.  Thus there exists an effective divisor $\Delta$ in $C$, not intersecting $E$ such that $\OO_P(\infty_P - 0_P +\pi^* \Delta - \pi^* (N \cdot E))$ has a global section. In other words, there is an effective divisor $F$ in $P$, necessarily vertical, such that
\[
\infty_P + \pi^* \Delta = 0_P + \pi^*(N \cdot E) + F.
\]

Let the curves $G_t$ be the members of the pencil spanned by $0_P+ \pi^*(N \cdot E) + F$ and $\infty_P + \pi^*\Delta$. All but finitely many members of this pencil are comb-like curves with handle a section of $P\rightarrow [C/G]$.  Since the base locus of the pencil contains $\infty_P\cap \pi^*(N \cdot E)$, these section curves agree with $\infty_P$ over $(N \cdot E)$. Restricting to $(N \cdot E)$ we get a one parameter family of pseudo-ideal sheaves in $N\cdot E$, hence also in $N \cdot D$ such that they agree with the given formal section $\wht{s}$ to a given order.

\end{proof}

With this family of pseudo-ideal sheaves in hand, all we need to do is to show that the obstruction groups of lifting the deformation to the Hilbert scheme vanish.

\subsection{Vanishing of the tangent obstruction}

 The sheaf
\[
\OO_{\mathfrak{X}'}(-N\cdot D)\cdot
\textit{Hom}_{\OO_{C_\kappa}}(\mc{I}_\kappa/\mc{I}_\kappa^2,\OO_{C_\kappa})
\]
is supported in the union of the section $s'(B')$ and the very free curves since $C_0$ is contained in $D$. If we attach enough very free curves, we can make
\[
H^1(C_\kappa,\OO_{\mathfrak{X}'}(-N \cdot D)\cdot
\textit{Hom}_{\OO_{C_\kappa}}(\mc{I}_\kappa/\mc{I}_\kappa^2,\OO_{C_\kappa}))
\]
zero (by Remark \ref{coherent sheaf}, applying the \emph{relative smoothing} Theorem \ref{relative smoothing}).

Now we want to show that
\[
H^1(C_\kappa, \textit{Tor}_{\OO_X}(\OO_{C_\kappa}, \OO_D)\cdot
\textit{Hom}_{\OO_{C_\kappa}}(\mc{I}_\kappa/\mc{I}_\kappa^2,\OO_{C_\kappa}))
\]
is also $0$. First, using the exact sequence
\[
0 \rightarrow \OO(-N\cdot D) \rightarrow \OO_{\mcX'} \rightarrow \OO_{N\cdot D} \rightarrow 0
\]
we see that
\[
\textit{Tor}_{\OO_{\mathfrak{X}'}}(\OO_{C_\kappa}, \OO_{N\cdot D})\cong \OO_{N\cdot C_0}(-p),
\]
where $p$ is the intersection of section and the stacky curve $C_0$ in $D$.  Thus the sheaf is supported on $C_0$. We have the short exact sequence of sheaves
\[
0 \rightarrow \mcN_{N \cdot C_0/\mathfrak{X}'}(-p) \rightarrow \mcN_{C_\kappa/\mathfrak{X}'}\rightarrow Q\rightarrow 0,
\]
where $Q$ is a torsion sheaf supported at $p$. Thus it suffices to show that
$$H^1(C_0, \mcN_{C_0/\mathfrak{X}'}(-p))=H^1(C_0, \mcN_{C_0/[\mcX'/G]}(-p))$$ is $0$. Let $C_0'$ and $p'$ be the preimages of $C_0$ and $p$ in $\mcX'$ and $q: C_0' \rightarrow C_0$ be the restriction of the quotient map $\mcX' \rightarrow [\mcX'/G]$. Then the sheaf
\[
\mcN_{C_0/[\mcX' /G]}(-p)
\]
is a direct summand of
\[
q_*(q^*\mcN_{C_0/[\mcX'/G]}(-p))=q_*(\mcN_{C_0'/\mcX'}(-p'))
\]
Therefore, the cohomology groups is also a direct summand of the corresponding cohomology. But we have chosen the curve $C_0'$ to be a very free curve in the fiber in the proof of Proposition \ref{prop:Requiv}. Thus we have $H^1(C_0' , N_{C_0'/\mcX'}(-p'))=0$ and so is
\[
H^1(C_\kappa, \textit{Tor}_{\OO_{\mathfrak{X}'}}(\OO_{C_\kappa}, \OO_D)\cdot
\textit{Hom}_{\OO_{C_\kappa}}(\mc{I}_\kappa/\mc{I}_\kappa^2,\OO_{C_\kappa})).
\]

 Therefore the divisor restriction map $\iota_{N\cdot D}:\text{Hilb}_{\mathfrak{X}'} \rightarrow \pis{N \cdot D}$ is smooth at $[C_\kappa]$. And we get a one parameter family of section curves whose restriction to the divisor $N\cdot D$ contains general members of the family of pseudo-ideal sheaves $\mc{I}_t$ in Lemma \ref{lem:ps} (c.f. Remark \ref{subtlety}).  In particular, we get a section of the actual family which agrees with the formal section to the given order $N$.

\begin{rem}\label{subtlety}
 As we noted in the beginning of this section, the space of pseudo-ideals is highly non-separated. Thus even if we know the divisor restriction map is dominant and there is a family of pseudo-ideal sheaves deforming from one to another, all we can conclude is that we have a family of sections whose restriction to the divisor is isomorphic to the family of pseudo-ideal sheaves at general points of the family. Therefore in our situations, it is essential to have a family of pseudo-ideal sheaves whose general members are all isomorphic to the pseudo-ideal sheaf coming from the formal section (e.g. the family $G_t$ coming from a pencil).
\end{rem}

\end{document}